%% file: main.tex
\documentclass[12pt]{article}
\input{Setup}
\usepackage[utf8]{inputenc}
\usepackage[hungarian,english]{babel}
\usepackage{blindtext}
\usepackage{t1enc}
\input{Mycommands}

\usepackage{cite}

\title{
Commutative Banach Algebras Generated by Toeplitz Operators on the Bergman Space}
\author{Miguel Angel Rodriguez Rodriguez
}

\begin{document}

\maketitle

\begin{abstract}
We present and study commutative Banach algebras
generated by Toeplitz operators with generalized quasi-radial
pseudo-homogeneous symbols
acting on the Bergman space over the unit ball.
We develop the Gelfand theory of these algebras
and give some structural information about them.
In particular,
we provide a description of the radical
of these algebras.
This paper generalizes and completes
the results from previous works
related to Toeplitz operators with
quasi-radial quasi-homogeneous symbols.
\vspace{1mm}

\end{abstract}
\input{Introduction}

\input{Pseudo-Introduction}
\input{Hardy-generalized}


\input{Pseudo-Gelfand}


\input{Spectral-invariance}

\input{Radical}


\input{ref}
\end{document}

%% file: Setup.tex
\usepackage[intlimits]{amsmath}
\usepackage{amsfonts,mathtools}
\usepackage{amsthm}
\usepackage{comment}
\usepackage{mathrsfs}
\usepackage{graphicx}
\usepackage[svgnames]{xcolor}
\usepackage{array}
\usepackage{pgfplots}
\pgfplotsset{compat=1.17}
\usepackage{xcolor}

\swapnumbers

\usepackage{tikz}
\usetikzlibrary{positioning, angles, arrows.meta, calc, quotes}
\usepackage{verbatim}
\usepackage{pgfplots}
\pgfplotsset{my style/.append style={axis x line=middle, axis y line=
middle, xlabel={$x$}, ylabel={$y$}, axis equal }}

%% file: Mycommands.tex
\theoremstyle{plain}
\newtheorem{prop}{Proposition}[section]
\newtheorem{cor}[prop]{Corollary}
\newtheorem{lem}[prop]{Lemma}
\newtheorem{thm}[prop]{Theorem}

\theoremstyle{definition}

\theoremstyle{plain}

\theoremstyle{definition}

\newcommand{\bC}{\mathbb{C}}
\newcommand{\bR}{\mathbb{R}}
\newcommand{\bZ}{\mathbb{Z}}

\newcommand{\bT}{\mathbb{T}}

\newcommand{\bB}{\mathbb{B}}

\newcommand{\cT}{\mathcal{T}}

\newcommand{\cj}[1]{\overline{#1}}
\newcommand{\iK}{\mathcal{K}}

\newcommand{\pdi}[1]{\langle #1 \rangle}

\newcommand{\spc}{\operatorname{sp}}
\newcommand{\spn}{\operatorname{span}}

\newcommand{\ka}{\kappa}
\newcommand{\cmp}[1]{#1^\mathrm{C}}
\newcommand{\Jt}{J_\theta}
\newcommand{\Jtc}{J_{\cmp{\theta}}}

\newcommand{\cA}{\mathcal{A}}

%% file: Introduction.tex
\section{Introduction}

Let $\bB^n$ denote the unit ball of $\bC^n$ and
for $\lambda>-1$ consider the Bergman space $\cA^2_\lambda(\bB^n)$,
defined as the subspace of analytic functions of $L^2(\bB^n,dv_\lambda)$, where
\[
dv_\lambda(z)=c_\lambda(1-|z|^2)^\lambda dv(z)
\]
is a probability measure with
$c_\lambda$ a normalizing constant given by \eqref{normalizing constant}
and $dv$ the standard Lebesgue measure.
Let $B_\lambda$ be the orthogonal projection from $L^2(\bB^n,dv_\lambda)$ onto the Bergman space $\cA^2_\lambda(\bB^n)$.
Given a bounded measurable function $\varphi\in L^\infty(\bB^n)$, we define the Toeplitz operator with symbol $\varphi$ by
\[
T_\varphi f=B_\lambda(\varphi f),\quad f\in\cA^2_\lambda(\bB^n).
\]

The study of algebras generated by Toeplitz operators have been an interesting and important part of mathematics in the recent years.
Since not much can be said about 
these algebras in all its generality, one usually considers Toeplitz operators whose symbols satisfy special properties.

As it turns out,
the underlying geometry of $\bB^n$ is related to the structure of Toeplitz algebras.
Indeed, as it was shown in
\cite{Vasilevski2007},
given a maximal abelian subgroup of biholomorphisms of $\bB^n$,
the set of all bounded measurable symbols invariant under this group
gives rise to commutative $C^*$ algebras on the Bergman space $\cA^2_\lambda(\bB^n)$ for every weight parameter $\lambda>-1$.

Afterwards and quite unexpectedly,
it was shown in \cite{Vasilevski2010,Vasilevski20102,Vasilevski2012,Vasilevski2017,Bauer2012h,Bauer2012n}
that some sets of symbols,
subordinated to the above maximal abelian subgroups, induce Banach algebras of operators that
are commutative for every
weighted Bergman space $\cA^2_\lambda(\bB^n)$ and such that the $C^*$ algebra generated by them is no longer commutative.

These algebras represent an interesting object of study and the main tool to obtain structural information from them is the Gelfand theory.
A first step towards the understanding of these algebras was done in \cite{Bauer2012,Bauer2013,Bauer2015},
where the authors studied Banach algebras
generated by Toeplitz operators with quasi-radial quasi-homogeneous symbols.

In this paper introduce a generalized version of these symbols and study the structure of the corresponding commutative Banach algebras by means of the Gelfand theory. In particular, we describe the radical of these algebras, a problem which had remained open in the previous works.

In Section \ref{sec: Preliminaries} we fix some notations, introduce the spaces and objects to be used and recall some important facts from previous works.
In Section \ref{sec: Gengeralized pseudo}
we consider a generalization of the quasi-homogeneous symbols introduced in \cite{Vasilevski2010} and describe the action of the induced Toeplitz operators.
The unital Banach algebra generated by these operators is studied in Section \ref{sec: The algebra Tc}.

In Section \ref{Pseudo-Gelfand} we present some of the main results of this paper.
We describe the maximal ideal space and the Gelfand transform of the corresponding algebra.

As an application of these results,
in Section \ref{sec: Spectral invariance} we briefly present a criterion to decide whether these algebras are spectral invariant
and in Section \ref{sec: Radical}
we describe the radical of these algebras,
generalizing and extending the results from \cite{Bauer2012,Bauer2013,Bauer2015}.

%% file: Pseudo-Introduction.tex
\section{Preliminaries}
\label{sec: Preliminaries}
Let $\bB^n$ be the unit ball of $\bC^n$.
Let $\cA_\lambda^2(\bB^n)$ be the classical
weighted Bergman space with parameter $\lambda>-1$
defined as the space of analytic functions of $L_2(\bB^n,dv_\lambda)$,
where
\[
dv_\lambda(z)=c_\lambda(1-|z|^2)^\lambda dv(z),
\]
is a probability measure with
$dv$ being the usual Lebesgue measure in $\bB^n$
and $c_\lambda$ a normalizing constant given by
\begin{equation}
\label{normalizing constant}
c_\lambda=\frac{\Gamma(n+\lambda+1)}{\pi^n\Gamma(\lambda+1)}.
\end{equation}
We will denote the inner product of $\cA_\lambda^2(\bB^n)$ by $\pdi{\cdot,\cdot}_{\cA_\lambda^2(\bB^n)}$ or simply $\pdi{\cdot,\cdot}$, if there is no confusion.

It is well known that $\cA_\lambda^2(\bB^n)$ is a
reproducing kernel Hilbert space and that the normalized monomials $(e_\alpha)_{\alpha\in\bZ_+^n}$ form an orthonormal basis, where
\begin{equation}\label{definition ealpha}
e_\alpha(z)=\sqrt{\frac{\Gamma(n+|\alpha|+\lambda+1)}{\alpha!\Gamma(n+\lambda+1)}}z^\alpha.
\end{equation}

We will simply write $\cA^2(\bB^n):=\cA^2_0(\bB^n)$ for the unweighted Bergman space.

Let $B_\lambda\colon L_2(\bB^n,dv_\lambda)\to \cA_\lambda^2(\bB^n)$
be the orthogonal projection called the \emph{Bergman projection}.
For a function $\varphi\in L_\infty(\bB^n)$, we define the
\emph{Toeplitz operator $T_\varphi$ with symbol $\varphi$} by
\[
T_\varphi(f)=B_\lambda(\varphi f),\quad f\in\cA_\lambda^2(\bB^n).
\]

For a given integer $m\in\{1,\ldots,n\}$
we choose a tuple
$k=(k_1,\ldots,k_m)\in\bZ_+^m$ with
$k_1\leq k_2\leq\ldots\leq k_m$ and $k_1+\cdots+k_m=n$.

We write $\bC^n=\bC^{k_1}\times\cdots\times\bC^{k_m}$
and let
\[
z=(z_{(1)},\ldots,z_{(m)})\in\bC^n,
\quad\text{with}\quad
z_{(j)}=(z_{j,1},\ldots,z_{j,k_j})\in\bC^{k_j}.
\]

In general, given any $n$-tuple $u$, we will write
\[
u=(u_1,\ldots,u_n)=(u_{(1)},\ldots,u_{(m)}),
\]
with
\[
u_{(j)}=(u_{j,1},\ldots,u_{j,k_j}),\quad j=1,\ldots,m,
\]
and
\begin{align*}
u_{1,1}&=u_1,\enspace u_{1,2}=u_2,\enspace\ldots,
\enspace u_{1,k_1}=u_{k_1}\\
u_{2,1}&=u_{k_1+1},\enspace u_{2,2}=u_{k_2+2},\enspace\ldots,
\enspace u_{2,k_2}=u_{k_1+k_2},\enspace\ldots,
\enspace u_{m,k_m}=u_{n}.
\end{align*}

We represent each coordinate of $z\in\bB^n$ in the form
\[
z_{j,l}=|z_{j,l}|t_{j,l},
\]
where $t_{j,l}$ belongs to the torus $\bT=S^1$.
For each portion $z_{(j)}$ of $z\in\bC^n$, $j=1,\ldots,m$
we introduce its ``common'' radius
\[
r_j=\sqrt{|z_{j,1}|^2+\cdots+|z_{j,k_j}|^2}
\]
and represent the coordinates of $z_{(j)}$ in the form
\[
z_{j,l}=r_js_{j,l}t_{j,l},\quad\text{where}\quad
l=1,\ldots,k_j
\]
and
\[
s_{(j)}=(s_{j,1},\ldots,s_{j,k_j})
\in S_+^{k_j-1},
\]
where $S_+^{k_j-1}=S^{k_j-1}\cap \bR_+^{k_j}$. That is, $S_+^{k_j-1}$
denotes the set of elements of the unit sphere in $\bR^{k_j}$
with non-negative entries.

A bounded measurable function $a=a(z)$, $z\in\bB^n$ is called
\emph{$k$-quasi-radial} if it depends only on $r_1,\ldots,r_m$.

For each $j\in\{1,\ldots,m\}$ let $c_j\in L_\infty(S_+^{k_j-1}\times\bT^{k_j})$ be a function invariant under the action of $\bT$ on $\bT^{k_j}$ given by $g\cdot t_{(j)}=(gt_{j,1},\ldots,gt_{j,k_j})$,
$g\in\bT,t_{(j)}\in\bT^{k_j}$. That is,
\begin{equation}\label{definition cj}
c_j(s_{(j)},g\cdot t_{(j)})=c_j(s_{(j)},t_{(j)}),\quad\forall (s_{(j)},t_{(j)})\in S_+^{k_j-1}\times\bT^{k_j},g\in\bT.
\end{equation}
We call such functions $c_j$ \emph{generalized pseudo-homogeneous}.
We note that this idea generalizes that of quasi-homogeneous and pseudo-homogeneous functions
presented in \cite{Vasilevski2010} and \cite{Vasilevski2017}, respectively.

Indeed, we obtain the pseudo-homogeneous functions by letting
$c_j(s_{(j)},t_{(j)})=b_j(s_{(j)})t^{p_{(j)}}$,
where $b_j$ is a bounded measurable function defined on $S_+^{k_j-1}$ and $p_{(j)}\in\bZ^{k_j}$ is such that
$|p_{(j)}|=p_{j,1}+\cdots +p_{j,k_j}=0$
obtain the quasi-homogeneous functions if, in addition, we let
\[
b_j(s_{(j)})=s_{j,1}^{|p_{j,1}|}\cdots s_{j,k_j}^{|p_{j,k_j}|}.
\]

\subsection{The algebra $\cT_\lambda(L_{k-qr}^\infty)$}
\label{Algebra T k-qr}

We recall and reproduce some facts and definitions from
\cite{Bauer2013},
\cite{Bauer2015}
and
\cite{Vasilevski2010}
that will be necessary.

The action of a Toeplitz operator with a $k$-quasi-radial symbol is given by the following lemma.
\begin{lem}[{{\cite[Lemma 3.1]{Vasilevski2010}}}]
Given a $k$-quasi-radial function $a=a(r_1,\ldots,r_m)$,
we have
\[
T_az^\alpha=\gamma_{a,k,\lambda}(\alpha)z^\alpha,\quad \alpha\in\bZ_+^n,
\]
where
\begin{align*}
\gamma_{a,k,\lambda}(\alpha)&=\widetilde{\gamma_{a,k,\lambda}}(|\alpha_{(1)}|,\ldots,|\alpha_{(m)}|)\\
&=\frac{2^m\Gamma(n+|\alpha|+\lambda+1)}{\Gamma(\lambda+1)\prod_{j=1}^m(k_j-1+|\alpha_{(j)}|)!}
\int_{\tau(\bB^m)}a(r_1,\ldots,r_m)(1-|r|^2)^\lambda\\
&\times\prod^m_{j=1}r_j^{2|\alpha_{(j)}|+2k_j-1}dr_j,
\end{align*}
where $\tau(\bB^m)$ denotes the base of $\bB^m$ as a Reinhardt domain, that is,
\[
\tau(\bB^m)=\{(r_1,\ldots,r_m)\in\bR_+^m\colon r_1^2+\cdots+r_m^2\leq1\}.
\]
\end{lem}

Let $\cT_\lambda(L_{k-qr}^\infty)$ be the $C^*$-algebra
generated by all Toeplitz operators with $k$-quasi-radial symbols.
This is an infinitely generated algebra, whose elements are diagonal operators with respect to the orthonormal basis
$\{e_\alpha\}_{\alpha\in\bZ_+^n}$ such that their eigenvalue sequences
depend only on the tuples $(|\alpha_{(1)}|,\ldots,|\alpha_{(m)}|)$.

By identifying each operator
in
$\cT_\lambda(L_{k-qr}^\infty)$
with its sequence of eigenvalues we may interpret this space
as an algebra
bounded sequences indexed by $\bZ_+^m$.
Moreover, in \cite{Bauer2012,Bauer2013} it was shown that
$\cT_\lambda(L_{k-qr}^\infty)$ can be embedded into a smaller
algebra of slowly oscillating (in a certain sense) sequences.

For each $\kappa=(\kappa_1,\ldots,\kappa_m)\in\bZ_+^m$
consider the finite dimensional subspace $H_\kappa$ of
$\cA_\lambda^2(\bB^n)$ defined by
\begin{equation}
\label{Definition Hkappa}
H_\kappa:=\spn\{e_\alpha\colon\alpha\in\bZ_+^n,\enspace
|\alpha_{(j)}|=\kappa_j,
\quad
j=1,\ldots,m\}.
\end{equation}
For each $j\in\{1,\ldots,m\}$ and $d\in\bZ_+$, we define
the subspace of $\cA_\lambda^2(\bB^n)$:
\[
H_d^{(j)}=\cj{\spn}\{e_\alpha\colon
|\alpha_{(j)}|=d\}\subset\cA_\lambda^2(\bB^n).
\]
The following relations hold, where all sums are orthogonal
\[
H_\ka=\bigcap_{j=1}^mH_{\ka_j}^{(j)},\quad
H_d^{(j)}=\bigoplus_{\ka_j=d}H_\ka,\quad\text{and}\quad
\cA_\lambda^2(\bB^n)=\bigoplus_{|\ka|=0}^\infty H_\ka.
\]
Let $P_\ka$ and $Q_d^{(j)}$ denote the orthogonal projections of
$\cA_\lambda^2(\bB^n)$ onto $H_\ka$ and $H_d^{(j)}$,
respectively. We note that
\[
P_\ka=\prod_{j=1}^mQ_{\ka_j}^{(j)}\quad\text{and}\quad
Q_d^{(j)}=\bigoplus_{\ka_j=d}P_\ka.
\]

\begin{thm}[{{\cite{Bauer2012,Bauer2013,Bauer2015}}}]
Let $\lambda>-1$. For all $\ka\in\bZ_+^m$,
$j\in\{1,\ldots,m\}$, and all $d\in\bZ_+$,
the orthogonal projections $P_\ka$ and $Q_d^{(j)}$
belong to the algebra $\cT_\lambda(L_{k-qr}^\infty)$.
\end{thm}

Let $M(\cT_\lambda(L_{k-qr}^\infty))$ be the compact set of maximal
ideals of $\cT_\lambda(L_{k-qr}^\infty)$.
We describe a stratification of $M(\cT_\lambda(L_{k-qr}^\infty))$ presented in \cite{Bauer2013,Bauer2015}.

Consider tuples $\theta=(\theta_1,\ldots,\theta_m)\in\{0,1\}^m$
with entries $0$ and $1$.
We write $\mathbf{1}:=(1,\ldots,1)\in\{0,1\}^m=:\Theta$, and denote by
$\cmp{\theta}=\mathbf{1}-\theta$ the tuple \emph{complementary}
to $\theta$.

Then, with $J_\theta=\{j\in\{1,\ldots,m\}\colon \theta_j=1\}$,
we introduce
\[
\bZ_+^\theta=\bigoplus_{j\in J_\theta}\bZ_+(j),
\]
where $\bZ_+(j)$ denotes a copy of $\bZ_+$ indexed by $j$.
Thus, its elements will be tuples of the form
\[
\ka_\theta=
(\ka_{j_1},\ldots,\ka_{j_{|\theta|}}),
\quad j_p\in J_\theta.
\]

For each $\theta\in\Theta$ we define a set $M_\theta$
of multiplicative functionals by
\[
M_\theta=
\left\{
\mu\in M(\cT_\lambda(L_{k-qr}^\infty))\colon
\mu(Q_d^{(j)})=
\begin{cases}
    0\text{ for all }&d\in\bZ_+,\text{ if }\theta_j=0.\\
    1\text{ for some }&d\in\bZ_+,\text{ if }\theta_j=1.
\end{cases}
\right\}.
\]

Thus $M_\theta$ is the set of all functionals in
$M(\cT_\lambda(L_{k-qr}^\infty))$ that are reached by nets
$\{\ka^\alpha\}_\alpha$ in $\bZ_+^m$ such that
$\ka^\alpha_j$ tends to infinity if and only if $\theta_j=0$.

All sets $M_\theta$ are disjoint and
\[
M(\cT_\lambda(L_{k-qr}^\infty))=\bigcup_{\theta\in\Theta}M_\theta.
\]
Moreover, $M_{\mathbf{1}}$ can be identified with $\bZ_+^m$ in a natural way,
and one can easily see that all functionals from
$M_\infty:=M(\cT_\lambda(L_{k-qr}^\infty))\backslash M_{\mathbf{1}}$
map compact operators of the algebra $\cT_\lambda(L_{k-qr}^\infty)$
to zero.

We recall as well that for each $\mu\in M_\theta$,
with $\theta\neq\mathbf{1}$,
there is a unique tuple $\ka_\theta\in\bZ_+^\theta$
(representing the "finite coordinates" of $\mu$)
such that $\mu(Q_{\ka_j}^{(j)})=1$ for all $j\in\Jt$.
This implies the following further decomposition of $M_\theta$
into disjoint sets
\[
M_\theta=\bigcup_{\ka_\theta\in\bZ_+^\theta}
M_\theta(\ka_\theta),
\]
where
$M_\theta(\ka_\theta)=
\{
\mu\in M_\theta\colon\mu(Q_{\ka_j}^{(j)})=1
\text{ for all }j\in\Jt
\}$.

%% file: Hardy-generalized.tex
\section{Generalized pseudo-homogeneous symbols}
\label{sec: Gengeralized pseudo}
For every $j\in\{1,\ldots,m\}$, let \[
d\mu_j=(2\pi)^{-k_j}\frac{dt_{j,1}}{it_{j,1}}\times\cdots\times\frac{dt_{j,k_j}}{it_{j,k_j}}
\]
be the invariant measure of $\bT^{k_j}$.
We also denote by $d\mu$ the invariant measure of $\bT^n$, which clearly satisfies
\[
d\mu=d\mu_1\times\cdots\times d\mu_{m}.
\]

\begin{lem}
Let $c_j$ be a generalized pseudo-homogeneous symbol as before. If $p_{(j)}\in\bZ^{k_j}$ with $|p_{(j)}|\neq0$ and $s_{(j)}\in S^{k_j-1}_+$, then
\[
\int_{\bT^{k_j}}c_j(s_{(j)},t_{(j)})t_{(j)}^{p_{(j)}}d\mu_j=0.
\]
\end{lem}
\begin{proof}
Let $g\in\bT$ be arbitrary. Consider the change of variables $t_{(j)}'=g\cdot t_{(j)}$. By the invariance of $c_j$ and of the measure we have
\[
\int_{\bT^{k_j}}
c_j(s_{(j)},t'_{(j)})
(t'_{(j)})^{p_{(j)}}d\mu_j(t'_{(j)})
=
g^{|p_{(j)}|}
\int_{\bT^{k_j}}
c_j(s_{(j)},t_{(j)})
(t_{(j)})^{p_{(j)}}
d\mu_j(t_{(j)}).
\]
Thus either the integrals are zero or $|p_{(j)}|=0$.
\end{proof}

Given $p_{(j)}\in\bZ^{k_j}$ we denote by $\widetilde{p_{(j)}}$ the element of $\bZ^n=\bZ^{k_1}\times\cdots\times\bZ^{k_m}$ given by
\[
\widetilde{p_{(j)}}=(0,\ldots,0,p_{(j)},0,\ldots,0).
\]
Here the tuple $p_{(j)}$ is located at the $j$-position.

For a positive integer $p$ we define the set
\[
\Delta_{p}=\{(r_1,\ldots,r_p)\in\bR_+^p\colon r_1+\cdots+r_p\leq1\}.
\]
Imitating the calculations from the proof of Lemma 3.2 from \cite{Vasilevski2017} and using the above lemma we describe the action of a Toeplitz operator with $k$-quasi-radial quasi-homogeneous
symbols:
\begin{prop}\label{action Tac}
Let $a$ be a $k$-quasi-radial function and $c=\prod_{j=1}^mc_j$. Then:
\begin{enumerate}
    \item For any $\alpha,\beta\in\bZ_+^n$,
$\pdi{T_{ac}(z^\alpha),z^\beta}=0$ if $|\alpha_{(j)}|\neq|\beta_{(j)}|$ for some $j\in\{1,\ldots,m\}$.

    \item For every $p\in\bZ^{n}$ with $|p_{(j)}|=0$, $j=1,\ldots,m$, we have
\begin{align*}
\pdi{&T_{ac}(z^\alpha),z^{\alpha+p}}\\
&=\frac{\Gamma(n+|\alpha|+\lambda+1)}{\Gamma(\lambda+1)\prod^m_{j=1}\Gamma(|\alpha_{(j)}|+k_j)}\\
&\times\int_{\Delta_m}a(\sqrt{r_1},\ldots,\sqrt{r_m})
\prod_{j=1}^mr_j^{|\alpha_{(j)}|+k_j-1}(1-(r_1+\ldots+r_m))^\lambda dr_1\cdots dr_m\\
&\times\prod_{j=1}^m
\frac{\Gamma(|\alpha_{(j)}|+k_j)}
{\prod_{\ell=1}^{k_j}\Gamma(\alpha_{j,\ell}+p_{j,\ell}+1)}
\int_{\Delta_{k_j-1}}\hat{c}(s^{1/2}_{(j)},p_{(j)})
\prod_{\ell=1}^{k_j-1}s_{j,\ell}^{\alpha_{j,\ell}+\frac{1}{2}p_{j,\ell}}\\
&\times (1-(s_{j,1}+\ldots+s_{j,k_j-1}))^{\alpha_{j,k_j}+\frac{1}{2}p_{j,k_j}}
ds_{j,1}\cdots ds_{j,k_j-1},
\end{align*}
where $s_{(j)}^{1/2}=(s_{j,1}^{1/2},\ldots,s_{j,k_j}^{1/2})$
and
$\hat{c}_j(s_{(j)},p_{(j)})$ is the $p_{(j)}$ Fourier coefficient of $c_j(s_{(j)},t_{(j)})$
given by
\[
\hat{c}_j(s_{(j)},p_{(j)})=\int_{\bT^{k_j}}c_j(s_{(j)},t_{(j)})t_{(j)}^{-p_{(j)}}d\mu_j(t_{(j)}),
\quad p_{(j)}\in\bZ^{k_j}.
\]
\end{enumerate}
\end{prop}

Fixing a symbol $c_j$ and letting the other ones be the constant function $1$ in the above proposition we get:
\begin{cor}\label{action Tcj}
Let $j\in\{1,\ldots,m\}$ and consider the Toeplitz operator $T_{c_j}$, where $c_j$ is given as before.
If $|\alpha_{(j)}|\neq|\beta_{(j)}|$ or $\alpha_{(j')}\neq \beta_{(j')}$,
then $\pdi{T_{c_j}z^\alpha,z^\beta}=0$. Otherwise,
for $\alpha\in\bZ_+^n$ and $p_{(j)}\in\bZ_+^{k_j}$
we have
\begin{align*}
\pdi{&T_{c_j}z^\alpha,z^{\alpha+\widetilde{p}_{(j)}}}\\
&
=\frac{\Gamma(|\alpha_{(j)}|+k_j)}
{\prod_{\ell=1}^{k_j}\Gamma(\alpha_{j,\ell}+p_{j,\ell}+1)}
\int_{\Delta_{k_j-1}}\hat{c}(s^{1/2}_{(j)},p_{(j)})
\prod_{\ell=1}^{k_j-1}s_{j,\ell}^{\alpha_{j,\ell}+\frac{1}{2}p_{j,\ell}}\\
&\times (1-(s_{j,1}+\ldots+s_{j,k_j-1}))^{\alpha_{j,k_j}+\frac{1}{2}p_{j,k_j}}
ds_{j,1}\cdots ds_{j,k_j-1}.
\end{align*}
\end{cor}

It follows from these expressions that for any $k$-quasi-radial function $a$, the operators $T_a$ and $T_{c_j}$, $j=1,\ldots,m$
pairwise commute and
\[
T_{a\prod_{j=1}^m c_j}=T_a\prod_{j=1}^mT_{c_j}.
\]

On the other hand, using the same notation, it follows from Propositions \ref{action Tac} and \ref{action Tcj}
that the operators $T_{c_j}$, acting on $\cA_\lambda^2(\bB^n)=\bigoplus_{\ka\in\bZ_+^m}^\infty H_\ka$,
can be decomposed as a direct sum
\[
T_{c_j}=\bigoplus_{\ka\in\bZ_+^m}T_{c_j}|_{H_\ka},
\]
whereas the operator $T_a$ acts as a constant multiple of the identity operator on each of these subspaces:
\[
T_a=\bigoplus_{\ka\in\bZ_+^m}\gamma_a(|\ka_1|,\ldots,|\ka_m|)I|_{H_\ka}.
\]

Furthermore, we note that by \eqref{definition ealpha},
\[
\|z^{\alpha+\widetilde{p_{(j)}}}\|/\|z^\alpha\|=
\sqrt{
(\alpha_{(j)}+p_{(j)})!/\alpha_{(j)}!
}.
\]
Thus, normalizing the monomials in Corollary \ref{action Tcj}, we see that the action of $T_{c_j}$
depends only on the $(j)$-portion of the involved elements.
That is, it does not depend on the other $k_j$'s nor on the weight parameter $\lambda$.
This suggests us to analyze these operators as follows.

Consider the tensor product of weightless Bergman spaces $\bigotimes_{j=1}^m\cA^2(\bB^{k_j})$ and let
\[
U\colon\cA^2_\lambda(\bB^n)\to\bigotimes_{j=1}^m\cA^2(\bB^{k_j})
\]
be the unitary operator such that
\[
U(e_\alpha)=
e^{(1)}_{\alpha_{(1)}}\otimes\ldots\otimes e^{(m)}_{\alpha_{(m)}},
\quad\alpha\in\bZ_+^n,
\]
where $e^{(j)}_{\alpha_{(j)}}$ denotes the canonical $\alpha_{(j)}$-basic monomial of $\cA^2(\bB^{k_j})$,
$j=1,\ldots,m$.

Since the symbols $c_j$ depend by definition only on the portion $z_{(j)}$,
we can associate to them in the obvious way
unique measurable functions defined on $\bB^{k_j}$.
To avoid notation we will write $c_j$ for both of these functions.

By the above results and observations, we obtain
\[
T_{c_j}=U^*(I\otimes\cdots\otimes\boldsymbol{T_j}\otimes\cdots\otimes I)U,
\]
where $\boldsymbol{T_j}$ denotes the Toeplitz operator with symbol $c_j$
acting on the Bergman space $\cA^2(\bB^{k_j})$.
In particular, we have
\[
\operatorname{sp}(T_{c_j})=\operatorname{sp}(\boldsymbol{T_j}).
\]

Moreover, if for every $d\in\bZ_+$ and $j\in\{1,\ldots,m\}$ we define $\boldsymbol{H}^{(j)}_{d}$ as the
finite-dimensional subspace of
$\cA^2(\bB^{k_j})$ given by
\[
\boldsymbol{H}^{(j)}_{d}=\operatorname{span}\{e^{(j)}_{\alpha_{(j)}}\colon |\alpha_{(j)}|=d\},
\]
then we also have
\begin{equation}\label{sum decomposition Tj}
\boldsymbol{T_j}=\bigoplus_{d\in\bZ_+}\boldsymbol{T_j}|_{\boldsymbol{H}^{(j)}_{d}},\quad j=1,\ldots,m.   
\end{equation}
Note that
$U(H^{(j)}_d)=\cA^2(\bB^{k_1})\otimes\cdots\boldsymbol{H}^{(j)}_{d}\otimes\cdots\otimes \cA^2(\bB^{k_m})$
and $U(H_{\ka})=\boldsymbol{H}^{(1)}_{\ka_1}\otimes\cdots\otimes \boldsymbol{H}^{(m)}_{\ka_m}$.

Let $\boldsymbol{d}=(\cT_\lambda(L^\infty_{k-qr}),c_1,\ldots,c_m)$ and let $\cT_\lambda(\boldsymbol{d})$ be the unital Banach algebra generated by the operators from $\cT_\lambda(L^\infty_{k,qr})$ and by $T_{c_1}$,\ldots,$T_{c_m}$. By our above observations, $\cT_\lambda(\boldsymbol{d})$ is a commutative Banach algebra
for every weight parameter $\lambda>-1$ and, as it was observed in the previous works
regarding quasi-homogeneous and pseudo-homogeneous symbols,
the corresponding $C^*$-algebras generated by $\cT_\lambda(\boldsymbol{d})$ is in general non-commutative, in general.

\section{The algebra $\cT(c_1,\ldots,c_m)$}
\label{sec: The algebra Tc}

Note that, by \eqref{sum decomposition Tj}, the operator $\boldsymbol{T_j}$ decomposes into a direct sum of operators acting on finite-dimensional spaces.
Therefore we can use the same arguments as in Section 3 from \cite{Rodriguez2021}
and obtain similar conclusions as the results therein. Thus we have:
\begin{prop}\label{Tcj spectrum 1}
The point spectrum of $\boldsymbol{T_j}$ is given by
\[
\operatorname{sp}_{pt}(\boldsymbol{T_j})=\displaystyle\bigcup_{d\in\bZ_+}\operatorname{sp}(\boldsymbol{T_j}|_{\boldsymbol{H}^{(j)}_{d}}).
\]
\end{prop}

Given a bounded set $S$ in $\bC^N$, we will denote its polynomially convex hull as $\widehat{S}$.

\begin{prop}\label{Tcj spectrum 2}
The spectrum of $\boldsymbol{T_j}$ is given by
\[
\operatorname{sp}(\boldsymbol{T_j})=\operatorname{sp}_{pt}(\boldsymbol{T_j})\cup\operatorname{ess-sp}(\boldsymbol{T_j}).
\]
and furthermore
\[
\widehat{\operatorname{sp}(\boldsymbol{T_j})}=
\operatorname{sp}_{pt}(\boldsymbol{T_j})\cup
\widehat{\operatorname{ess-sp}}(\boldsymbol{T_j}).
\]
\end{prop}

\begin{prop}\label{Tcj spectrum 3}
If $\zeta$ is an accumulation point of $\spc_{pt}(\boldsymbol{T_j})$,
then $\zeta\in\operatorname{ess-sp}(\boldsymbol{T_j})$.
\end{prop}

For each $j=1,\ldots,m$, we will assume that for every
$\zeta\in\operatorname{ess-sp}(\boldsymbol{T_j})$ there is a sequence of unimodular vectors $(g^{(j)}_{\zeta,k})_k$ such that $g^{(j)}_{\zeta,k}\in \boldsymbol{H}^{(j)}_{k}$ and
\begin{equation}\label{condition cj}
\lim_{k\to\infty}\|(\boldsymbol{T_j}-\zeta I)g_{\zeta,k}^{(j)}\|=0.    
\end{equation}

This condition holds
whenever the functions $c_j$ can be continuously extended to the closed unit ball $\cj{\bB^{k_j}}$.
We describe briefly how to see this.
Note that this construction simplifies the calculations from Section 3 of \cite{Bauer2015} (for the case $m=1$),
where the reproducing kernel of a much larger space was considered.

For each $j\in\{1,\ldots,m\}$ and $d\in\bZ_+$
let $K^{(j)}_d(z,w)$ denote the reproducing kernel of $\mathbf{H}^{(j)}_d$.
By the multinomial theorem we have
\[
K^{(j)}_d(z,w)=\sum_{|\alpha_{(j)}|=d}e_{\alpha_{(j)}}^{(j)}(z)\overline{e_{\alpha_{(j)}}^{(j)}(w)}
=
\frac{(k_j+d)!}{k_j!d!}\pdi{z,w}^d,
\quad z,w\in\bB^{k_j},
\]
and thus $\|K^{(j)}_d(\cdot,w)\|^2=\frac{(k_j+d)!}{k_j!d!}|w|^{2d}$.
We denote the corresponding normalized reproducing kernels by
\[
k^{(j)}_d(z,w)=\frac{K^{(j)}_d(z,w)}{\|K^{(j)}_d(\cdot,w)\|},
\quad z,w\in\bB^{k_j}.
\]

\begin{prop}
\label{Berezin transform}
Let $c$ be a continuous function on $\overline{\bB^{k_j}}$ invariant under the action of $\bT$,
(that is, such that $c(g\cdot z)=c(z)$, for all $z\in\overline{\bB^{k_j}}$, $g\in\bT$).
Then for every $w\in\bB^{k_j}\backslash\{0\}$ we have
\[
c(w/|w|)=\lim_{d\to\infty}\pdi{T_ck^{(j)}_d(\cdot,w),k^{(j)}_d(\cdot,w)},
\]
where $T_c$ denotes the Toeplitz operator with symbol $c$ acting on $\cA^2(\bB^{k_j})$.
In particular
\[
\lim_{d\to\infty}\|(T_c-c(w/|w|)I)k^{(j)}_d(\cdot,w)\|=0.
\]
\end{prop}
\begin{proof}
It follows from similar arguments as in \cite[Propositions 3.5 and 3.6]{Rodriguez2021},
estimating the corresponding integrals in $\bB^{k_j}$ instead of $\bB^2$.
\end{proof}

If $c_j$ can be continuously extended to $\overline{\bB^n}$, then
$\operatorname{ess-sp}(\boldsymbol{T_j})=c_j(\partial\bB^n)$, and thus by the above proposition we get condition \eqref{condition cj}, as claimed.

For each $j\in\{1,\ldots,m\}$ let $\cT_j$ be the unital Banach algebra generated by the single operator $T_{c_j}$
and $M(\cT_j)$ its maximal ideal space. Analogously let $\cT(c_1,\ldots,c_m)$ the Banach algebra generated by $\cT_1,\ldots,\cT_{m}$ and let
$M(\cT(c_1,\ldots,c_m))$ be its maximal ideal space.
Note that the mapping
\[
M(\cT(c_1,\ldots,c_m))\ni\psi\longmapsto
(\psi|_{\cT_1},\ldots,\psi|_{\cT_m})\in
M(\cT_1)\times\cdots\times M(\cT_m)
\]
is a continuous injection and that thus $M(\cT(c_1,\ldots,c_m))$
can be identified with a subset of $M(\cT_1)\times\cdots\times M(\cT_m)$.
Reasoning as in Section 4.1 from \cite{Rodriguez2021},
one can easily see that $M(\cT_j)=\widehat{\operatorname{sp}(T_{c_j}})$.

\begin{cor}\label{finite D}
Let $j\in\{1,\ldots,m\}$ and $\zeta\in
\widehat{\operatorname{sp}(\boldsymbol{T_j})}
\backslash
\widehat{\operatorname{ess-sp}}(\boldsymbol{T_j})$. Then the set
\[
D=\{d\in\bZ_+\colon \zeta\in \operatorname{sp}(\boldsymbol{T_j}|_{\boldsymbol{H}^{(j)}_{d}})\}
\]
is finite.
\end{cor}
\begin{proof}
If $D$ were not finite, then there would be an increasing sequence of integers
$(n_k)_k$ and a sequence of
normalized eigenfunctions $(f_{n_k})_k$ with $\boldsymbol{T_j}f_{n_k}=\zeta f_{n_k}$
and $f_{n_k}\in\boldsymbol{H}^{(j)}_{n_k}$, for all $k\in\bZ_+$.
However, this would imply
$\zeta\in \operatorname{ess-sp}(\boldsymbol{T_j})$.
\end{proof}

\begin{prop}\label{maximal ideal space Tc1cm}
We have
\[
M(\cT(c_1,\ldots,c_m))= M(\cT_1)\times\cdots\times M(\cT_m)
=
\widehat{\operatorname{sp}(T_{c_1})}\times\cdots\times\widehat{\operatorname{sp}(T_{c_m})}.
\]
\end{prop}
\begin{proof}
We only need to prove the inclusion $\supset$.
By the properties of the polynomially convex hull we have
\[
\widehat{\operatorname{sp}(T_{c_1})\times\cdots\times\operatorname{sp}(T_{c_m})}=\widehat{\operatorname{sp}(T_{c_1}})\times\cdots\times\widehat{\operatorname{sp}(T_{c_m})}.
\]
Since $M(\cT(c_1,\ldots,c_m))$ is polynomially convex,
it suffices to show that
\[
\operatorname{sp}(T_{c_1})\times\cdots\times\operatorname{sp}(T_{c_m})\subset M(\cT(c_1,\ldots,c_m)).
\]
Let thus $(\zeta_1,\ldots,\zeta_m)\in\operatorname{sp}(T_{c_1})\times\cdots\times\operatorname{sp}(T_{c_m})$.

For each $j=1,\ldots,m$ define the sequence $(g^{(j)}_{\zeta_j,k})_k$ in $\mathcal{A}_\lambda(\bB^{k_j})$ as follows: If $\zeta_j\in\operatorname{ess-sp}(\boldsymbol{T_j})$, let $(g^{(j)}_{\zeta_j,k})_k$ be a sequence such that \eqref{condition cj} holds.
Otherwise, $\zeta_j\in\operatorname{sp}_{pt}(\boldsymbol{T_j})$ and thus then there is $d\in\bZ_+$ such that
$\zeta_j\in\operatorname{sp}_{pt}(\boldsymbol{T_j}|_{\boldsymbol{H}^{(j)}_{d}})$.
In this case we choose $(g^{(j)}_{\zeta_j,k})_k$ to be a constant sequence $g^{(j)}_{\zeta_j,k}=v_{\zeta_j}$, where $v_{\zeta_j}$ is a eigenvector in $\boldsymbol{H}^{(j)}_{d}$ such that $\|v_{\zeta_j}\|=1$ and $\boldsymbol{T_j}v_{\zeta_j}=\zeta_jv_{\zeta_j}$.

Since $UT_{c_j}U^*=I\otimes\cdots\otimes\boldsymbol{T_j}\otimes\cdots\otimes I$, we have
\[
\pdi{UT_{c_j}U^*(g^{(1)}_{\zeta_1,k}\otimes\cdots\otimes g^{(m)}_{\zeta_m,k}),g^{(1)}_{\zeta_1,k}\otimes\cdots\otimes g^{(m)}_{\zeta_m,k}}=\pdi{\boldsymbol{T_j}g^{(j)}_{\zeta_j,k},g^{(j)}_{\zeta_j,k}}_{\cA^2(\bB^{k_j})},
\]
which converges to $\zeta_j$ (in fact it equals $\zeta_j$ if $\zeta_j\in\operatorname{sp}_{pt}(\boldsymbol{T_j})$) as $k\to\infty$.

Therefore, for any polynomial $T$ in $T_{c_1},\ldots,T_{c_m}$
we can define $\psi(T)$ by the rule
\[
\psi(T)=\lim_{k\to\infty}\pdi{UTU^*(g^{(1)}_{\zeta_1,k}\otimes\cdots\otimes g^{(m)}_{\zeta_m,k}),g^{(1)}_{\zeta_1,k}\otimes\cdots\otimes g^{(m)}_{\zeta_m,k}}.
\]
By Cauchy-Schwartz inequality, this function can be extended to
a multiplicative functional acting on the whole space $\cT(c_1,\ldots,c_m)$
and since $\psi(T_j)=\zeta_j$, this functional corresponds to the point $(\zeta_1,\ldots,\zeta_m)$.
\end{proof}

%% file: Pseudo-Gelfand.tex
\section{Gelfand theory of $\cT_\lambda(\mathbf{d})$}
\label{Pseudo-Gelfand}

Let $\widetilde{\cT_\lambda(\mathbf{d})}$ be the
dense non-closed subalgebra of $\cT_\lambda(\mathbf{d})$
which consists of all sums of the form
\[
\sum_{\rho\in F}D_{\gamma_\rho}
T_{c_1}^{\rho_1}\cdots
T_{c_m}^{\rho_m},
\]
where
$D_{\gamma_\rho}$ are operators from
$\cT_\lambda(L^\infty_{k-qr})$ and
$F$ is a finite subset of $\bZ_+^{m}$.
For the sake of brevity, given a multi-index $\rho\in\bZ_+^m$ we will write
$T^\rho:=T_{c_1}^{\rho_1}\cdots
T_{c_m}^{\rho_m}$. Furthermore, we will continuously write $\zeta$ for $(\zeta_1,\ldots,\zeta_m)$.

Let $M(\cT_\lambda(\mathbf{d}))$ be the maximal ideal space
of the algebra $\cT_\lambda(\mathbf{d})$.
By construction,
$\cT_\lambda(\mathbf{d})$
is generated by the algebras
$\cT_\lambda(L_{k-qr}^\infty)$ and $\cT(c_1,\ldots,c_m)$ and hence,
as in the previous section, the assignment
\[
\Phi\colon
M(\cT_\lambda(\mathbf{d}))
\to
M(\cT_\lambda(L_{k-qr}^\infty))\times
M(\cT(c_1,\ldots,c_m)),
\]
given by
\[
\Phi(\psi)=(\psi|_{\cT_\lambda(L_{k-qr}^\infty)},
\psi|_{\cT(c_1,\ldots,c_m)})
\]
is continuous and injective,
and thus we can identify $M(\cT_\lambda(\mathbf{d}))$
with a subset of
\begin{align*}
M(\cT_\lambda(L_{k-qr}^\infty))&\times
M(\cT(c_1,\ldots,c_m))=\\
&=
M(\cT_\lambda(L_{k-qr}^\infty))\times\widehat{\operatorname{sp}(T_{c_1}})\times\cdots\times\widehat{\operatorname{sp}(T_{c_m})}.
\end{align*}

We recall the decomposition of
$M(\cT_\lambda(L_{k-qr}^\infty))$
into disjoint sets
\begin{equation}\label{decomposition Mkqr}
M(\cT_\lambda(L_{k-qr}^\infty))=
\bigcup_{\theta\in\Theta}
\left[
\bigcup_{\kappa_\theta\in\bZ_+^{\theta}}
M_\theta(\kappa_\theta)
\right].   
\end{equation}

\begin{lem}\label{maximal ideal space second inclusion finite}
Let $\psi=(\mu,\zeta)\in M(\cT_\lambda(\mathbf{d}))$ such that
$\mu\in M_\theta(\ka_\theta)$ for some
$\theta\in\Theta$ and $\ka_\theta\in\bZ_+^\theta$.
Then
$\zeta_j\in\operatorname{sp(\mathbf{T_j}|_{\mathbf{H}^{(j)}_{\ka_j}}})$ for every
$j\in\Jt$.
\end{lem}
\begin{proof}
According to Section \ref{Algebra T k-qr},
we have $\psi(Q_d^{(j)})=\mu(Q_d^{(j)})=0$ for every
index $j$ and $d\in\bZ_+$, unless
$\theta_j=1$ and $d=\ka_j$,
in which case $\psi(Q_d^{(j)})=\mu(Q_d^{(j)})=1$.

Let $j\in\Jt$ and let $P$ be the characteristic polynomial of $\mathbf{T_j}|_{\mathbf{H}^{(j)}_{\ka_j}}$.
Then $P(T_{c_j})Q_{\ka_j}^{(j)}=0$, and hence
we have
\[
P(\psi(T_{c_j}))=\psi(P(T_{c_j}))=\psi(P(T_{c_j})Q_{\ka_j}^{(j)})=0,
\]
from which it follows that
\[
\psi(T_{c_j})\in\operatorname{sp(\mathbf{T_j}|_{\mathbf{H}^{(j)}_{\ka_j}}}).\qedhere
\]
\end{proof}

\begin{lem}\label{maximal ideal space second inclusion infty}
Let $\psi=(\mu,\zeta)\in M(\cT_\lambda(\mathbf{d}))$ such that
$\mu\in M_\theta(\ka_\theta)$ for some
$\theta\in\Theta$ and $\ka_\theta\in\bZ_+^\theta$.
Then
$\zeta_j\in\widehat{\operatorname{ess-sp}(T_{c_j})}$ for every
$j\in\Jtc$.
\end{lem}
\begin{proof}
Suppose there is some $j\in\Jtc$ such that
$\zeta_j\in M(\cT_j)\backslash\widehat{\operatorname{ess-sp}(T_{c_j})}$.
By Proposition \ref{Tcj spectrum 2} and Corollary \ref{finite D},
this means that $\zeta_j\in\operatorname{sp}_{pt}(\mathbf{T_j})$
and the set
\[
D=\{d\in\bZ_+\colon \zeta_j\in\operatorname{sp}(\mathbf{T_j}|_{\mathbf{H}_d^{(j)}})\}
\]
is finite. Consider the orthogonal projection $P=\bigoplus_{d\in D}Q^{(j)}_d$
and note that $\psi(P)=0$.

Let $w\in\bC$ be such that $\zeta_j\notin\bigcup_{d\in D}\operatorname{sp}(\mathbf{T_j}|_{\mathbf{H}_d^{(j)}})+w$.
One easily sees that
\begin{align*}
\widehat{\operatorname{sp}(T_{c_j}+wP)}&=
\widehat{\operatorname{ess-sp}}(T_{c_j}+wP)
\cup\operatorname{sp}_{pt}(T_{c_j}+wP)\\&=
\widehat{\operatorname{ess-sp}}(T_{c_j})
\cup
\bigcup_{d\in\bZ_+\backslash D}
\operatorname{sp}(\mathbf{T_j}|_{\mathbf{H}_d^{(j)}})
\cup
\left[
\bigcup_{d\in D}
\operatorname{sp}(\mathbf{T_j}|_{\mathbf{H}_d^{(j)}})+w
\right],
\end{align*}
and so $\zeta_j\notin\widehat{\operatorname{sp}(T_{c_j}+wP)}$.
But then we have the contradiction
\[
\zeta_j=\psi(T_{c_j})=\psi(T_{c_j}+wP)\in\widehat{\operatorname{sp}(T_{c_j}+wP)}.\qedhere
\]
\end{proof}

\begin{lem}\label{multiplicative functionals representation}
Let $\theta\in\Theta$ and $\ka_\theta\in\bZ_+^\theta$.
Consider the sets $D_j$ given by
\[
D_j=
\begin{cases}
\spc(\mathbf{T_j}|_{\mathbf{H}_{\ka_j}^{(j)}}),\quad &
j\in\Jt,\\
\operatorname{ess-sp}(\mathbf{T_j}),\quad &
j\in\Jtc.
\end{cases}
\]
Then we have $M_\theta(\ka_\theta)\times D_1\times\cdots\times D_m\subset M(\cT_\lambda(\mathbf{d}))$.

Moreover, if $\psi\in M_\theta(\ka_\theta)\times D_1\times\cdots\times D_m$,
then there is a net $(g_\alpha)_\alpha$ such that
\[
\lim_\alpha\|(T-\psi(T)I)g_\alpha\|=0.
\]
\end{lem}
\begin{proof}
Let $\psi=(\mu,\zeta_1,\ldots,\zeta_m)\in M_\theta(\ka_\theta)\times D_1\times\cdots\times D_m$.
Consider a net $(\ka_\alpha)_\alpha$ in $\bZ_+^m$ such that $\mu=\lim_\alpha\ka_\alpha$.
Suppose that
$\Jt=\{j_1,\ldots,j_{|\theta|}\}$ and
$\ka_\theta=(\ka_{j_1},\ldots,\ka_{j_{|\theta|}})$.

Since $\zeta_j\in\operatorname{sp}(\mathbf{T}_j)$ for each $j$, we can define the sequences $(f^{(j)}_{\zeta_j,k})_k$ as in the proof of Proposition \ref{maximal ideal space Tc1cm}.
If $j\in\Jt$, we can assume that
$(\ka_\alpha)_j=\ka_j$ for all $\alpha$
 and that $f^{(j)}_{\zeta_j,k}\in\mathbf{H}_{(\ka_\theta)_j}^{(j)}$ for each $k$.

Thus we have
\[
\lim_{k\to\infty}\|(\mathbf{T_j}-\zeta_j)f^{(j)}_{\zeta_j,k}\|=0,
\]
for each $j=1,\ldots,m$.

Using the same directed set of indices $\alpha$ we define a net $(g_\alpha)_\alpha$ by
\[
g_\alpha=U^*(f^{(1)}_{\zeta_1,\ka^1_\alpha}\otimes\cdots\otimes f^{(m)}_{\zeta_m,\ka^m_\alpha})\in H_{\ka_\alpha}.
\]

Then we have
\[
\lim_\alpha\|(T_{c_j}-\zeta_j)g_\alpha\|=0,\quad j=1,\ldots,m,
\]
and, since $D_\gamma(g_\alpha)=\gamma(\ka_\alpha)g_\alpha$,
\[
\lim_\alpha\|(D_\gamma-\mu)g_\alpha\|=
\lim_\alpha|\gamma(\ka_\alpha)-\mu|=0.
\]

Thus, by these properties, for an operator $A\in\widetilde{\cT_\lambda(\mathbf{d})}$ of the form
$A=\sum_{\rho\in F}D_{\gamma_\rho}T^\rho$
we get
\[
\lim_\alpha\|(A-\psi(A))g_\alpha\|=
\lim_\alpha\|(A-\sum_{\rho\in F_A}\gamma_\rho(\ka_\alpha)\zeta^\rho)g_\alpha\|=0.
\]
We can extend this formula
to the whole algebra $\cT_\lambda(\mathbf{d})$
and, in particular, we can define the functional
\[
\psi(T)=
\lim_\alpha\pdi{Tg_\alpha,g_\alpha},\quad T\in\cT_\lambda(\mathbf{d}),
\]
obtaining thus a multiplicative functional, corresponding to the point $(\mu,\zeta)$.
\end{proof}

\begin{thm}
\label{Maximal ideal space Tdx}
The compact set of maximal ideals
$M(\cT_\lambda(\mathbf{d}))$ of the commutative
Banach algebra
$\cT_\lambda(\mathbf{d})$ coincides with the set
\[
M:=
\bigcup_{\theta\in\Theta}
\left[
\bigcup_{\kappa_\theta\in\bZ_+^{\theta}}
M_\theta(\kappa_\theta)
\times
M_{\theta,\ka_\theta,1}\times\cdots\times M_{\theta,\ka_\theta,m}
\right],
\]
where
\[
M_{\theta,\ka_\theta,j}=
\begin{cases}
\spc(\mathbf{T_j}|_{\mathbf{H}_{\ka_j}^{(j)}}),\quad &
j\in\Jt,\\
\widehat{\operatorname{ess-sp}}(\mathbf{T_j}),\quad &
j\in\Jtc.
\end{cases}
\]

The Gelfand transform is generated by the following mapping
of elements of $\widetilde{\cT_\lambda(\mathbf{d})}$:
\begin{align*}
A:=\sum_{\rho\in F_A}D_{\gamma_\rho}
T^\rho\longmapsto
\sum_{\rho\in F_A}
D_\gamma(\mu)\zeta^\rho.
\end{align*}
\end{thm}
\begin{proof}
We prove the inclusion $M\subset M(\cT_\lambda(\mathbf{d}))$.
The other inclusion follows from Lemmas \ref{maximal ideal space second inclusion finite} and \ref{maximal ideal space second inclusion infty}.

Fix some $\theta\in\Theta$ and $\ka_\theta\in\bZ_+^\theta$
and consider the sets $D_1,\ldots,D_m$ as in Lemma \ref{multiplicative functionals representation}.
By this lemma,
$M_\theta(\ka_\theta)\times D_1\times\cdots \times D_m\subset M(\cT_\lambda(\mathbf{d}))$.

Write $D=D_1\times\cdots\times D_m$. Then we clearly have $\widehat{D_j}=M_{\theta,\ka_\theta,j}$
and $\widehat{D}=\widehat{D_1}\times\cdots\times\widehat{D_m}$,
and thus we need to prove that $M_\theta(\ka_\theta)\times\widehat{D}\subset M(\mathcal{T}_\lambda(\mathbf{d}))$.

Let $(\mu,\zeta)$ be an element of the left-hand side set in this inclusion.
By the definition of the polynomial convex hull and Lemma \ref{multiplicative functionals representation},
for any operator $A\in\widetilde{\cT_\lambda(\mathbf{d})}$ of the form
$A=\sum_{\rho\in F_A}D_{\gamma_\rho}T^\rho$ we have
\[
|\sum_{\rho\in F_A}\gamma_\rho(\mu)\zeta^\rho|
\leq
\sup_{z\in D}|\sum_{\rho\in F_A}\gamma_\rho(\mu)z^\rho|
\leq
\sup_{\varphi\in M_\theta(\ka_\theta)\times D}|\varphi(A)|
\leq \|A\|.
\]

Therefore, there is a well-defined bounded multiplicative functional $\psi$ on $\widetilde{\cT_\lambda(\mathbf{d})}$ given by
\[
\psi(A)=\sum_{\rho\in F_A}\gamma_\rho(\mu)\zeta^\rho,\quad
A=\sum_{\rho\in F_A}D_{\gamma_\rho}T^\rho\in\widetilde{\cT_\lambda(\mathbf{d})}.
\]
This functional extends then to a multiplicative functional on $\cT_\lambda(\mathbf{d})$, which corresponds to the point $(\mu,\zeta_1,\ldots,\zeta_m)$.
\end{proof}

%% file: Spectral-Invariance.tex
\section{Spectral invariance}
\label{sec: Spectral invariance}

Let $\mathcal{B}$ be a Banach algebra of operators acting on $\cA^2_\lambda(\bB^n)$.
We say that $\mathcal{B}$ is \emph{spectral invariant} if each operator from $\mathcal{B}$ invertible in $\mathcal{L}(\cA^2_\lambda(\bB^n))$ is invertible in $\mathcal{B}$.
In this case, the spectrum of an operator $T$ in $\mathcal{B}$, denoted by $\sigma_{\mathcal{B}}(T)$,
equals its spectrum in the $C^*$-algebra $\mathcal{L}(\cA^2_\lambda(\bB^n))$.

\begin{prop}
The algebra $\cT_\lambda(\boldsymbol{d})$ is inverse closed if and only if
$\operatorname{sp}(T_{c_j})$ is polynomially convex for every $j=1,\ldots,m$.

Thus, in this case,
\[
\sigma_{\cT_\lambda(\boldsymbol{d})}(T)
=
\operatorname{sp}(T),
\]
for every $T\in\cT_\lambda(\boldsymbol{d})$.
\end{prop}
\begin{proof}
Suppose that $\operatorname{sp}(T_{c_j})$ is polynomially convex for every $j=1,\ldots,m$.
Let $T\in\cT_\lambda(\boldsymbol{d})$ be invertible in the $C^*$-algebra $\mathcal{L}(\cA^2_\lambda(\bB^n))$ and let $T^{-1}$ be its inverse.

By Lemma \ref{multiplicative functionals representation} and Theorem \ref{Maximal ideal space Tdx},
for every $\psi\in M(\cT_\lambda(\boldsymbol{d}))$ there is a net of normalized vectors $(g_\alpha)_\alpha$
such that
\[
\lim_\alpha\|(T-\psi(T)I)g_\alpha\|=0,\quad T\in\cT_\lambda(\boldsymbol{d}).
\]

Thus we have
\[
1=\pdi{T^{-1}Tg_\alpha,g_\alpha}=
\psi(T)\pdi{T^{-1}g_\alpha,g_\alpha}+\pdi{T^{-1}(T-\psi(T)I)g_\alpha,g_\alpha},
\]
where the second summand converges to zero.
In particular, $\psi(T)\neq0$ and, thus, $T$ is invertible in $\cT_\lambda(\boldsymbol{d})$.

Conversely, suppose that $\cT_\lambda(\boldsymbol{d})$ is spectral invariant.
Then
\[
\sigma_{\cT_\lambda(\boldsymbol{d})}(T_{c_j})=\operatorname{sp}(T_{c_j}),
\]
but, by Theorem \ref{Maximal ideal space Tdx},
$\sigma_{\cT_\lambda(\boldsymbol{d})}(T_{c_j})=\widehat{\operatorname{sp}(T_{c_j})}$.
\end{proof}

%% file: Radical.tex
\section{The radical of $\mathcal{T}_\lambda(\mathbf{d})$}
\label{sec: Radical}

In this section we will assume that $c_j$,
regarded as a function defined on $\bB^{k_j}$,
can be continuously extended to the closed unit ball of $\overline{\mathbb{B}^{k_j}}$.
By Proposition \ref{Berezin transform} and the remarks after it, this stronger condition implies \eqref{condition cj}.
Furthermore, as is well known, in this case the operators $\mathbf{T_j}$ acting on $\mathcal{A}^2(\mathbb{B}^{k_j})$ commute with their respective adjoints $\mathbf{T_j}^*$ modulo a compact operator.

Given $j\in\{1,\ldots,m\}$ and $\ka\in\bZ_+^m$ we will write
\[
\widetilde{Q}_{\ka_j}^{(j)}=\bigoplus_{0\leq d\leq \ka_j}Q^{(j)}_d.
\]

Thus $\widetilde{Q}_{\ka_j}^{(j)}$ maps the Bergman space $\mathcal{A}_\lambda^2(\mathbb{B}^{n})$
onto $\bigoplus_{\rho\in\bZ_+^m,0\leq \rho_j\leq \ka_j}H_\rho$.
Moreover, we note that
\[
UQ_d^{(j)}U^*=I\otimes\cdots\otimes\mathbf{Q}^{(j)}_d\otimes\cdots\otimes I,
\]
where $\mathbf{Q}^{(j)}_d$ is the orthogonal projection from $\mathcal{A}^2(\mathbb{B}^{k_j})$ onto $\mathbf{H}^{(j)}_d$.
Similary we have
\[
U\widetilde{Q}^{(j)}_{\ka_j}U^*=
I\otimes\cdots\otimes\widetilde{\mathbf{Q}}^{(j)}_{\ka_j}\otimes\cdots\otimes I,
\]
where
$\widetilde{\mathbf{Q}}^{(j)}_{\ka_j}$ is the orthogonal
projection from $\mathcal{A}^2(\mathbb{B}^{k_j})$ onto $\bigoplus_{0\leq d\leq \ka_j}\mathbf{H}^{(j)}_d$.

In particular, the projections $\widetilde{\mathbf{Q}}^{(j)}_{\ka_j}$ strongly converge to the identity operator on $\mathcal{A}^2(\mathbb{B}^{k_j})$,
as $\ka_j\to\infty$,
and if $T$ is an operator in $\cA^2_\lambda(\bB^n)$
such that, for some compact operator $K\in\iK(\mathcal{A}^2(\mathbb{B}^{k_j}))$,
we have
$UTU^*=I\otimes\cdots\otimes K\otimes\cdots\otimes I$,
being $K$ in the $j$-position,
then $\widetilde{Q}^{(j)}_{\ka_j}T$ converges in norm to $T$,
as $\ka_j\to\infty$.

Given $\{j_1,\ldots,j_L\}\subset\{1,\ldots,m\}$ and the projections
$\widetilde{Q}_{d_{j_1}}^{(j_1)},\ldots,\widetilde{Q}_{d_{j_L}}^{(j_L)}$,
we define
\begin{equation}
\label{projections QS def}
P_1=\widetilde{Q}_{d_{1}}^{(j_1)}
\end{equation}
and inductively
\begin{equation}
\label{projections QS def2}
P_{\ell+1}=\widetilde{Q}_{d_{j_{\ell+1}}}^{(j_{\ell+1})}-\widetilde{Q}_{d_{j_{\ell+1}}}^{(j_{\ell+1})}
(P_1+\cdots+P_{\ell}),
\quad\ell=1,\ldots,L-1.
\end{equation}

\begin{lem}\label{projections QS}
The projections $P_\ell$, $\ell=1,\ldots,L$,
are mutually orthogonal
and the space generated by the union of its ranges equals the space
generated by the union of the ranges of
$\widetilde{Q}_{d_{j_1}}^{(j_1)},\ldots,\widetilde{Q}_{d_{j_L}}^{(j_L)}$.
\end{lem}
\begin{proof}
It follows from the fact that
if $Q_1$ and $Q_2$ are commuting orthogonal projections acting on a Hilbert space,
then $Q_1$ and $Q_2-Q_2Q_1$ are mutually orthogonal projections and
the space generated by the union of its ranges equal the space generated
by the union of the ranges of $Q_1$ and $Q_2$.
\end{proof}

It will be useful to consider algebras defined similarly as the algebra
$\cT_\lambda(\mathbf{d})$, but "omitting" some of the generators $T_{c_j}$.
We formalize this as follows.

Let $\theta\in\Theta$.
We denote by $\mathcal{T}_\lambda^{\theta}$ the Banach algebra generated by $\mathcal{T}_\lambda(L^\infty_{k-qr})$ and the operators $T_{c_j}$ if $\theta_j=1$.
We also denote by $\widetilde{\mathcal{T}}_\lambda^{\theta}$
the non-closed dense subalgebra generated by those elements,
consisting of finite sums of finite products of the generators.

We will denote by $\epsilon_j$ the $m$-tuple with its $j$-th entry equal $1$,
being $0$ the other ones. So, given $\theta\in\Theta$ and $j\in\Jt$,
the tuple $\theta-\epsilon_j$ equals $\theta$ except that its $j$-th entry is zero.

Let $\mathcal{J}^\theta$ be the closed $*$-ideal in $\mathcal{L}(\mathcal{A}_\lambda^2(\mathbb{B}^{n}))$ generated by all operators of the form
$U^*(I\otimes \cdots\otimes K_j\otimes\cdots\otimes I)U$, where $j\in\Jt$
and $K_j$ is a compact operator acting on $\mathcal{A}^2(\mathbb{B}^{k_j})$.
Since we are assuming the symbols $c_j$ to be continuous at the boundary of $\bB^{k_j}$,
we have
$[T_{c_j},T_{c_j}^*]\in\mathcal{J}^\theta$ for every $j\in\Jt$.

Consider now the $C^*$-algebra
$(\mathcal{T}_\lambda^{\theta})^*$ generated by $\mathcal{T}_\lambda^{\theta}$. By construction, the $C^*$-algebra
\[
\widehat{(\mathcal{T}_\lambda^{\theta})^*}:=(\mathcal{T}_\lambda^{\theta})^*/\mathcal{J}^\theta
\]
is commutative. Reasoning as in the remarks preceding Lemma 6.11 from \cite{Bauer2015}, consider the canonical projection
\[
\pi^\theta\colon
\mathcal{T}_\lambda^{\theta}
\longrightarrow
\mathcal{T}_\lambda^{\theta}/\mathcal{J}^\theta\cap\mathcal{T}_\lambda^{\theta}
\]
and note that
\[
\mathcal{T}_\lambda^{\theta}/\mathcal{J}^\theta\cap\mathcal{T}_\lambda^{\theta}
\cong
(\mathcal{T}_\lambda^{\theta}+\mathcal{J}^\theta)/\mathcal{J}^\theta
\subset
((\mathcal{T}_\lambda^{\theta})^*+\mathcal{J}^\theta)/\mathcal{J}^\theta
\cong
\widehat{(\mathcal{T}_\lambda^{\theta})^*}.
\]

As in the proof of that lemma,
the multiplicative functionals from $\widehat{(\mathcal{T}_\lambda^{\theta})^*}$
can be applied to the elements of $\mathcal{T}_\lambda^{\theta}$ via $\pi^\theta$,
from which we conclude:
\begin{lem}\label{Rad subset J}
It holds
\[
\operatorname{Rad}(\mathcal{T}_\lambda^{\theta})
\subset
\mathcal{J}^\theta\cap\mathcal{T}_\lambda^{\theta}.
\]
\end{lem}

\medskip

Given $j\in\{1,\ldots,m\}$ and $d\in\bZ_+$, let $n_{j,d}=|\operatorname{sp}(\mathbf{T_j}|_{\mathbf{H}^{(j)}_{d}})|$ and
suppose that $\operatorname{sp}(\mathbf{T_j}|_{\mathbf{H}^{(j)}_{d}})=\{\zeta^{j,d}_1,\ldots,\zeta^{j,d}_{n_{j,d}}\}$
(we are considering the eigenvalues without repetition).

We define the polynomials
\[
h^{j,d}_{1}(X)=X-\zeta^{j,d}_1,
\]
\[
h^{j,d}_{2}(X)=(X-\zeta^{j,d}_1)(X-\zeta^{j,d}_2),
\]
\[
\vdots
\]
\[
h^{j,d}_{n_{j,d}}(X)=(X-\zeta^{j,d}_1)\cdots(X-\zeta^{j,d}_{n_{j,d}}).
\]

We can use these polynomials to decide whether the algebra $\cT_\lambda(\mathbf{d})$ is semi-simple.

\begin{prop}
The algebra $\cT_\lambda(\mathbf{d})$ is semi-simple if and only if
every matrix $\mathbf{T_j}|_{\mathbf{H}^{(j)}_d}$, $j\in\{1,\ldots,m\}$, $d\in\bZ_+$,
is diagonalizable in the sense
that its Jordan canonical form is diagonal.
\end{prop}
\begin{proof}
Suppose every matrix $\mathbf{T_j}|_{\mathbf{H}^{(j)}_d}$ is as stated.
If there is $0\neq T\in \operatorname{Rad}(\mathcal{T}_\lambda(\mathbf{d}))$,
then there is $\ka\in\bZ_+^m$ such that
$T|_{H_\ka}\neq 0$.

Our hypothesis implies that for every $j\in\{1,\ldots,m\}$ there is a linear (not necessarily orthonormal) basis for $\mathbf{H}^{(j)}_{\ka_j}$ formed by normalized eigenvectors of $\mathbf{T_j}$.
Thus the corresponding tensor products form a linear basis for $U(H_\ka)$.
One can easily see that for every such tensor product
$v$ and an operator $S\in\mathcal{T}_\lambda(\mathbf{d})$ we have
$S(U^*(v))=\psi(S)U^*(v)$,
where $\psi$ is the multiplicative functional defined by
\[
S\longmapsto\pdi{S( U^*(v)),U^*(v)},
\]
corresponding to some point
$(\ka,\zeta)\in M(\cT_\lambda(\mathbf{d}))$,
where $\zeta\in \operatorname{sp}_{pt}
(\mathbf{T_1}|_{\mathbf{H}^{(1)}_{\ka_1}})\times\cdots\times \operatorname{sp}_{pt}(\mathbf{T_m}|_{\mathbf{H}^{(m)}_{\ka_m}})$.
Since $T|_{H_\ka}\neq 0$, there is such a tensor $v$ with $T(U^*(v))\neq0$, but then we have the contradiction $T( U^*(v))=\psi(T)U^*(v)=0$.

Conversely, if there are some $j\in\{1,\ldots,m\}$ and $d\in\bZ_+$ such that
$\mathbf{T_j}|_{\mathbf{H}^{(j)}_d}$ fails to be diagonalizable,
then by Theorem \ref{Maximal ideal space Tdx}
\[
\operatorname{Rad}(\mathcal{T}_\lambda(\mathbf{d}))\ni Q^{(j)}_dh^{j,d}_{n_{j,d}}(\mathbf{T_j}|_{\mathbf{H}^{(j)}_d})\neq 0.
\]
\end{proof}

As in previous works, one step towards the description of the radical is the description of the subalgebra $\widetilde{\cT_\lambda(\mathbf{d})}\cap \operatorname{Rad}(\mathcal{T}_\lambda(\mathbf{d}))$.

Taking an element from this algebra and supposing that it is mapped to zero by all multiplicative functionals from Theorem \ref{Maximal ideal space Tdx},
one sees that it lies in the non-closed ideal generated by operators of the form
\begin{equation}
\label{typical element radical Bergman}
D_\gamma
\bigoplus_{d\in F_L}
Q^{(j)}_d
h^{j,d}_{n_{j,d}}(T_{c_j}),
\end{equation}
where 
$D_\gamma\in \cT_\lambda(L^\infty_{k-qr})$ is such that $\gamma(\mu)=0$ for any $\mu\in M_\theta$
with $\theta_j=0$ and $F_L=\{d\in\bZ_+\colon
|\operatorname{sp}
(\mathbf{T_j}
|_{\mathbf{H}^{(j)}_d})
|
\leq L\}$,
$L\in\bZ_+$.

Since $F_L$ could be infinite,
it is not clear in general whether such an operator always belong to $\cT_\lambda(\mathbf{d})$,
however it does for all known cases.
We present some of these:
\begin{prop}
\label{radical dense intersection Bergman}
Assume that one of the following conditions holds:
\begin{enumerate}
    \item $m=1$,
    \item All matrices $\mathbf{T_j}|_{\mathbf{H}^{(j)}_{d}}$ are nilpotent,
    \item For each $j\in\{1,\ldots,m\}$, the number of distinct eigenvalues $n_{j,d}$ of the matrix $\mathbf{T_j}|_{\mathbf{H}^{(j)}_{d}}$ tends to infinity as $d\to\infty$. 
\end{enumerate}
Then
$\widetilde{\cT_\lambda(\mathbf{d})}\cap \operatorname{Rad}(\mathcal{T}_\lambda(\mathbf{d}))$
is the non-closed algebra generated by all operators of the form
\eqref{typical element radical Bergman}.
\end{prop}
\begin{proof}
We prove the first case.
The two other cases follow from similar calculations.

If $m=1$, then we have only one generator $T:=\mathbf{T_1}=T_{c_1}$ and $\mathbf{H}^{(1)}_d=H^{(1)}_d=H_d$. Let $A\in\widetilde{\cT_\lambda(\mathbf{d})}\cap \operatorname{Rad}(\mathcal{T}_\lambda(\mathbf{d}))$ and suppose that
$A=\sum_{\ell=0}^L
D_{\gamma_\ell}T^\ell$.
For $d\in\bZ_+$ we note that
$A|_{H_d}$ is of the form
$A=\sum_{\ell=0}^L \gamma_\ell(d) T^\ell$.

Let $(d,\zeta)\in M(\cT_\lambda(\mathbf{d}))$ with $d\in\bZ_+$. By Theorem \ref{Maximal ideal space Tdx} we have $\zeta\in\operatorname{sp}(\mathbf{T_1}|_{\mathbf{H}_d})$ and thus $0=(d,\zeta)(A)=\sum_{\ell=0}^L\gamma_\ell(d)\zeta^\ell$ imply that
$\sum_{\ell=0}^L
\gamma_\ell(d) X^\ell=
p_{d}(X)h^{1,d}_{n_{1,d}}(X)$,
for some polynomial $p_d$.
Note that $p_d(X)=0$ whenever $n_{1,d}>L$
and so we can write
\[
A=\bigoplus_{d\in\bZ_+}Q^{(1)}_dA
=\bigoplus_{d\in\bZ_+}Q^{(1)}_d
p_{d}(T)h^{1,d}_{n_{1,d}}(T)
=
p(T)\bigoplus_{d\in F_L}Q^{(1)}_d
h^{1,d}_{n_{1,d}}(T)|_{H_d},
\]
where
$F_L=\{d\in\bZ_+\colon
|\operatorname{sp}
(T
|_{{H}_d})
|
\leq L\}$
and $p(T)=\bigoplus_{d\in F_L}p_d(T)$ can be written as a is a finite sum of finite products of $T$ and diagonal operators not necessarily belonging to
$\cT_\lambda(L^\infty_{k-qr})$
but still constant on all subspaces $H_d$.

Finally, applying the condition
$(\mu,\zeta)(A)=0$ for all functionals with $\mu\in M_{(0)}=M(\cT_\lambda(L^\infty_{k-qr}))\backslash\bZ_+$ and $\zeta\in\widehat{\operatorname{ess-sp}}(T)$
(according to Theorem \ref{Maximal ideal space Tdx}),
one easily sees that indeed
$p(T)\in\widetilde{\cT_\lambda(\mathbf{d})}$,
being a sum of operators of the form $D_\gamma T^\ell$,
where each $D_\gamma$ is compact.
\end{proof}

We show now that, as in all previously known cases, $\widetilde{\cT_\lambda(\mathbf{d})}\cap \operatorname{Rad}(\mathcal{T}_\lambda(\mathbf{d}))$ is dense in $\operatorname{Rad}(\mathcal{T}_\lambda(\mathbf{d}))$.
We remark that this holds in general and not only for the cases where
$\widetilde{\cT_\lambda(\mathbf{d})}\cap\operatorname{Rad}(\mathcal{T}_\lambda(\mathbf{d}))$ is explicitly known.



\begin{lem}\label{polynomial decomposition}
Let $A\in\widetilde{\mathcal{T}}_\lambda^{\theta}$.
Then there are operators 
$S^{j,d}_{n_{j,d}}(A)\in\widetilde{\mathcal{T}}_\lambda^{\theta}$,
and $S^{j,d}_{\ell}(A)\in\widetilde{\mathcal{T}}_\lambda^{\theta-\epsilon_j}$
(that is, $T_{c_j}$ does not appear in this polynomial expression),
$\ell=0,1,\ldots,n_{j,d}-1$,
such that $S^{j,d}_{n_{j,d}}(A)h^{j,d}_{n_{j,d}}(T_{c_j})\in\widetilde{\mathcal{T}}_\lambda^{\theta}\cap \operatorname{Rad}(\mathcal{T}_\lambda^{\theta})$
and
\begin{align}
\label{polynomial decomposition formula}
Q^{(j)}_dA&=
S^{j,d}_{n_{j,d}}(A)h^{j,d}_{n_{j,d}}(T_{c_j})\nonumber\\
&+S^{j,d}_{n_{j,d}-1}(A)h^{j,d}_{n_{j,d}-1}(T_{c_j})\nonumber\\
&\cdots\\
&+S^{j,d}_{1}(A)h^{j,d}_{1}(T_{c_j})\nonumber\\
&+S^{j,d}_{0}(A).\nonumber
\end{align}
\end{lem}
\begin{proof}
Suppose that $A=\sum_{\rho\in F}D_{\gamma_\rho}T^\rho$, with $F$ a finite subset of $\bZ_+^m$.
Then $A=h((D_{\gamma_\rho})_{\rho\in F},T_{c_1},\cdots,T_{c_m})$,
where $h$ is the polynomial in
$|F|+m$ variables
\[
h((Y_\rho)_{\rho\in F},X_1,\ldots,X_m)=
\sum_{\rho\in F}Y_{\rho}X^\rho,
\]
where $X^\rho=X_1^{\rho_1}\cdots X_m^{\rho_m}$.
(We note that, indeed, $h$ does not depend on the variables $X_j$ for
$j\notin\Jt$).
By applying several times the division algorithm for polynomials in the variable $X_j$, we can
construct unique functions $g^{j,d}_{0},g^{j,d}_{1},\ldots,g^{j,d}_{n_{j,d}}$
such that
\begin{align*}
h((Y_\rho)_{\rho\in F},X_1,\ldots,X_m)&=
g^{j,d}_{n_{j,d}}((Y_\rho)_{\rho\in F},X_1,\ldots,X_m)
h^{j,d}_{n_{j,d}}(X_j)\\
&+
g^{j,d}_{n_{j,d}-1}((Y_\rho)_{\rho\in F},X_1,\ldots,\widehat{X_j},\ldots,X_m)
h^{j,d}_{n_{j,d}-1}(X_j)\\
&+
\cdots\\
&+
g^{j,d}_{1}((Y_\rho)_{\rho\in F},X_1,\ldots,\widehat{X_j},\ldots,X_m)
h^{j,d}_{1}(X_j)\\
&+g^{j,d}_{0}((Y_\rho)_{\rho\in F},X_1,\ldots,\widehat{X_j},\ldots,X_m).
\end{align*}
(Here $\widehat{X_j}$ indicates that $X_j$ does not appear on the corresponding expression).
By successively evaluating this expression in the points of $\operatorname{sp}(\mathbf{T_j}|_{\mathbf{H}^{(j)}_{d}})=\{\zeta^{j,d}_1,\ldots,\zeta^{j,d}_{n_{j,d}}\}$ we find that the functions
$g^{j,d}_{\ell}$ are polynomials for $0\leq\ell\leq n_{j,d}-1$.
In particular, $g^{j,d}_{n_{j,d}}((Y_\rho)_{\rho\in F},X_1,\ldots,X_m)
h^{j,d}_{n_{j,d}}(X_j)$ is also a polynomial. Furthermore,
one easily sees that $g^{j,d}_{n_{j,d}}((Y_\rho)_{\rho\in F},X_1,\ldots,X_m)$
is bounded on each compact set of $\bC^{|F|+m}$ and, being a rational function,
it is then also a polynomial.

We obtain the formula from the lemma by evaluating
this polynomials in the corresponding operators
letting
$S^{j,d}_{\ell}(A)=Q^{(j)}_dg^{j,d}_\ell((D_{\gamma_\rho})_{\rho\in F},T_{c_1},\cdots,T_{c_m})$,
for $\ell=0.\ldots,n_{j,d}$.

Finally, by Theorem \ref{Maximal ideal space Tdx},
one can see that
$S^{j,d}_{n_{j,d}}(A)h^{j,d}_{n_{j,d}}(T_{c_j})\in\widetilde{\mathcal{T}}_\lambda^{\theta}\cap \operatorname{Rad}(\mathcal{T}_\lambda^{\theta})$.
\end{proof}

\begin{lem}
\label{norm tensor product}
Consider a tensor product of Hilbert spaces $H_1\otimes \cdots\otimes H_M$
and let $j\in\{1,\ldots,M\}$.
Let $w$ and $w'$ be fixed vectors in $H_j$
with $\|w\|=\|w'\|$, let
$s$ be a sum of tensors of the form
\[
v^1\otimes\cdots \otimes v^{j-1}\otimes w\otimes v^{j+1}\otimes\cdots\otimes v^M
\]
and let
$s'$ be equal $s$ but changing in every summand $w$ by $w'$.
Then $\|s\|=\|s'\|$.
\end{lem}
\begin{proof}
Suppose that $s=\sum_{i\in I} v^1_i\otimes\cdots \otimes v^{j-1}_i\otimes w\otimes v^{j+1}_i\otimes\cdots\otimes v^M_i$
and
$s'=\sum_{i\in I}v^1_i\otimes\cdots \otimes v^{j-1}_i\otimes w'\otimes v^{j+1}_i\otimes\cdots\otimes v^M_i$. Then
\begin{align*}
    \|s\|^2
    &=
    \sum_{i\in I,i'\in I}
    \pdi{v^1_i,v^1_{i'}}
    \cdots
    \pdi{w,w}
    \cdots
    \pdi{v^M_i,v^M_{i'}}
    \\
    &=
    \sum_{i\in I,i'\in I}
    \pdi{v^1_i,v^1_{i'}}
    \cdots
    \pdi{w',w'}
    \cdots
    \pdi{v^M_i,v^M_{i'}}=\|s'\|^2.\qedhere
\end{align*}
\end{proof}

We recall that the image of the projection $Q^{(j)}_d$ is
the closed span of all basic elements $e_\alpha$ such that $|\alpha_{(j)}|=d$
and that, by construction, we have $e_\alpha=U^*(e^{(1)}_{\alpha_{(1)}}\otimes\cdots\otimes e^{(m)}_{\alpha_{(m)}})$,
for every $\alpha\in\bZ_+^n$.

For a function $g\in\mathbf{H}^{(j)}_d$ with $\|g\|=1$ we define
the space $H^{(j)}(g)$ as the closed span of all functions of the form
\[
U^*(e^{(1)}_{\alpha_{(1)}}\otimes\cdots\otimes e^{(j-1)}_{\alpha_{(j-1)}}
\otimes g \otimes e^{(j+1)}_{\alpha_{(j+1)}}
\otimes\cdots\otimes e^{(m)}_{\alpha_{(m)}}).
\]
Thus, for example,
$H^{(j)}_d=\bigoplus_{|\alpha_{(j)}|=d}H^{(j)}(e^{(j)}_{\alpha_{(j)}})$.

For a fixed tuple $\alpha_{(j)}\in\bZ_+^{k_j}$ we define
the operator
\[
R^{(j)}_{\alpha_{(j)},g}\colon
H^{(j)}(e^{(j)}_{\alpha_{(j)}})
\longrightarrow
H^{(j)}(g)
\]
given by the extension of the assignment
\[
R^{(j)}_{\alpha_{(j)},g}(e_\beta)=
U^*(e^{(1)}_{\beta_{(1)}}\otimes\cdots\otimes e^{(j-1)}_{\beta_{(j-1)}}
\otimes g \otimes e^{(j+1)}_{\beta_{(j+1)}}
\otimes\cdots\otimes e^{(m)}_{\beta_{(m)}}),
\]
for every $\beta\in\bZ_+^n$ such that $\beta_{(j)}=\alpha_{(j)}$.
We note that, by Lemma \ref{norm tensor product}, we have $\|R^{(j)}_{\alpha_{(j)},g}(f)\|=\|f\|$
for every $f\in H^{(j)}(e^{(j)}_{\alpha_{(j)}})$.

Moreover, every $f\in H^{(j)}_d$ can be written in the form of an orthogonal sum
$f=\sum_{\substack{|\alpha_{(j)}|=d}}f_{\alpha_{(j)}}$,
where $f_{\alpha_{(j)}}\in H^{(j)}(e^{(j)}_{\alpha_{(j)}})$.
Hence we have
\[
\|f\|^2=\sum_{\substack{|\alpha_{(j)}|=d}}\|f_{\alpha_{(j)}}\|^2=
\sum_{\substack{|\alpha_{(j)}|=d}}\|R^{(j)}_{\alpha_{(j)},g}(f_{\alpha_{(j)}})\|^2.
\]

\begin{lem}
Using the notations from Lemma \ref{polynomial decomposition},
the assignments $A\longmapsto S^{j,d}_\ell(A)$ from $\widetilde{\mathcal{T}}_\lambda^{\theta}$ into $Q^{(j)}_d\widetilde{\mathcal{T}}_\lambda^{\theta-\epsilon_j}$
are linear functions
for $\ell=0,\ldots,n_{j,d}-1$
and there are positive constants $C^{j,d}_\ell$ such that
\[
\|S^{j,d}_\ell(A)\|\leq C^{j,d}_\ell\|A\|,
\quad
A\in
\widetilde{\mathcal{T}}_\lambda^{\theta},
\quad
\ell=0,\ldots,n_{j,d}-1.
\]

Thus the functions $S^{j,d}_\ell$ can be extended to the closed algebra
$Q^{(j)}_d\mathcal{T}_\lambda^{\theta}$.
In particular, the assignment $A\longmapsto S^{j,d}_{n_{j,d}}(A)h^{j,d}_{n_{j,d}}(T_{c_j})$
is also a linear continuous function, which can also be extended to this algebra.
\end{lem}
\begin{proof}

We proceed by induction on $\ell$.
Let $g\in\mathbf{H}^{(j)}_d$ be a unimodular eigenvector such that
$\mathbf{T_j}g=\zeta^{j,d}_1g$
and $f\in H^{(j)}_d$
and we write as before
$f=\sum_{|\alpha_{(j)}|=d}f_{\alpha_{(j)}}$,
where $f_{\alpha_{(j)}}\in H^{(j)}(e^{(j)}_{\alpha_{(j)}})$,
for every $\alpha_{(j)}$ with $|\alpha_{(j)}|=d$.

By construction we have
\[
h^{j,d}_{\ell}(T_{c_j})R^{(j)}_{\alpha_{(j)},g}(f_{\alpha_{(j)}})
=0,\quad
\ell=1,\ldots,n_{j,d}
\]
and thus by \eqref{polynomial decomposition formula} we get
\[
S^{j,d}_0(A)(R^{(j)}_{\alpha_{(j)},g}(f_{\alpha_{(j)}}))=
Q^{(j)}_dA R^{(j)}_{\alpha_{(j)},g}(f_{\alpha_{(j)}}).
\]

Moreover, since the operator $T_{c_j}$ does not appear in the expression of $S^{j,d}_0(A)$,
we observe that $S^{j,d}_0(A)(f_{\alpha_{(j)}})\in H^{(j)}(e^{(j)}_{\alpha_{(j)}})$ and
\[
R^{(j)}_{\alpha_{(j)},g}(S^{j,d}_0(A)(f_{\alpha_{(j)}}))=S^{j,d}_0(A)(R^{(j)}_{\alpha_{(j)},g}(f_{\alpha_{(j)}})).
\]
In particular, by Lemma \ref{norm tensor product},
\[
\|S^{j,d}_0(A)(f_{\alpha_{(j)}})\|=\|S^{j,d}_0(A)(R^{(j)}_{\alpha_{(j)},g}(f_{\alpha_{(j)}}))\|
=\|Q^{(j)}_dA R^{(j)}_{\alpha_{(j)},g}(f_{\alpha_{(j)}})\|.
\]

We set
\[
C^{j,d}_0=\operatorname{dim}(\mathbf{H}^{(j)}_d)^{1/2}=
\left|\left\{\alpha_{(j)}\in\bZ_+^{k_j}\colon|\alpha_{(j)}|=d\right\}\right|^{1/2}.
\]
By the Cauchy-Schwarz inequality we have
\begin{align*}
\|S^{j,d}_0(A)f\|^2&=\|\sum_{|\alpha_{(j)}|=d}S^{j,d}_0(A)(f_{\alpha_{(j)}})\|^2\\
&\leq
(C^{j,d}_0)^2\sum_{|\alpha_{(j)}|=d}\|S^{j,d}_0(A)(f_{\alpha_{(j)}})\|^2\\
&=
(C^{j,d}_0)^2\sum_{|\alpha_{(j)}|=d}\|Q^{(j)}_dAR^{(j)}_{\alpha_{(j)},g}(f_{\alpha_{(j)}})\|^2\\
&\leq
(C^{j,d}_0)^2\|A\|^2\sum_{|\alpha_{(j)}|=d}\|R^{(j)}_{\alpha_{(j)},g}(f_{\alpha_{(j)}})\|^2\\
&=
(C^{j,d}_0)^2\|A\|^2\sum_{|\alpha_{(j)}|=d}\|f_{\alpha_{(j)}}\|^2\\
&=
(C^{j,d}_0)^2\|A\|^2\|f\|^2.
\end{align*}
Since $f$ was arbitrary, we conclude $\|S^{j,d}_0(A)\|\leq C^{j,d}_0 \|A\|$.
In particular, the assignment $A\mapsto S^{j,d}_0(A)$ is a well-defined
linear and bounded function.

Now suppose we have constructed the constants $C^{j,d}_{\ell'}$
with the desired properties for $\ell'=0,1,\ldots,\ell<n_{j,d}-1$.
Using the decomposition from Lemma \ref{polynomial decomposition},
we define the application $S'$ acting on $Q^{(j)}_d\widetilde{\mathcal{T}}_\lambda^{\theta}$
by
\begin{align*}
S'(A)&=Q^{(j)}_dA-S_0^{j,d}(A)-\cdots-S_{\ell}^{j,d}(A)h^{j,d}_{\ell}(T_{c_j})\\
&=
S^{j,d}_{n_{j,d}}(A)h^{j,d}_{n_{j,d}}(T_{c_j})
+\cdots+
S^{j,d}_{\ell+1}(A)h^{j,d}_{\ell+1}(T_{c_j}), 
\end{align*}
for every $A\in Q^{(j)}_d\widetilde{\mathcal{T}}_\lambda^{\theta}$.

By induction hypothesis, $S'$ is continuous with
$\|S'\|\leq 1+C^{j,d}_0+\cdots+C^{j,d}_{\ell}$.
Reasoning similarly as before, let $g\in\mathbf{H}^{(j)}_d$ be an unimodular eigenvector
such that
$\mathbf{T_j}g=\zeta^{j,d}_{\ell+2}g$
and $f\in H^{(j)}_d$ with $f=\sum_{|\alpha_{(j)}|=d}f_{\alpha_{(j)}}$, as above.

Since
\[
Q^{(j)}_dS'(A)
R^{(j)}_{\alpha_{(j)},g}
(f_{\alpha_{(j)}})
=
S^{j,d}_{\ell+1}(A)
(\zeta^{j,d}_{\ell+2}-\zeta^{j,d}_{\ell+1})
\cdots
(\zeta^{j,d}_{\ell+2}-\zeta^{j,d}_1)
R^{(j)}_{\alpha_{(j)},g}
(f_{\alpha_{(j)}})
\]
and the operator $T_{c_j}$ does not appear in the expression of $S^{j,d}_{\ell+1}(A)$,
we can use similar arguments as in the case $\ell=0$
and obtain
\begin{align*}
\|S^{j,d}_{\ell+1}(A)f\|^2&\leq
(C^{j,d}_{0})^2|\zeta^{j,d}_{\ell+2}-\zeta^{j,d}_{\ell+1}|^{-2}
\cdots
|\zeta^{j,d}_{\ell+2}-\zeta^{j,d}_1|^{-2}\\
&\times
\sum_{|\alpha_{(j)}|=d}\|Q^{(j)}_dS'(A)(R^{(j)}_{\alpha_{(j)},g}
(f_{\alpha_{(j)}}))\|^2\\
&\leq
(C^{j,d}_{\ell+1})^2\|A\|^2\|f\|^2,
\end{align*}
where
\[
C^{j,d}_{\ell+1}=
|\zeta^{j,d}_{\ell+2}-\zeta^{j,d}_{\ell+1}|^{-1}
\cdots
|\zeta^{j,d}_{\ell+2}-\zeta^{j,d}_1|^{-1}
C_0^{j,d}
\left[
\left(
1+C^{j,d}_{0}+\cdots+C^{j,d}_{\ell}
\right)
\right]^{1/2}.\qedhere
\]
\end{proof}

The next corollary helps us understand what the operators from
the radical $\operatorname{Rad}(\mathcal{T}_\lambda^{\theta})$ look like on the spaces $H^{(j)}_d$.

\begin{cor}\label{polynomial decomposition 2}
If $T\in\operatorname{Rad}(\mathcal{T}_\lambda^{\theta})$,
then there are unique operators 
$S(T)\in\operatorname{clos}(\widetilde{\mathcal{T}}_\lambda^{\theta}\cap \operatorname{Rad}(\mathcal{T}_\lambda^{\theta}))$
and $S^{j,d}_{\ell}(T)\in\operatorname{Rad}(\mathcal{T}_\lambda^{\theta-\epsilon_j}), \ell=0,\ldots,n_{j,d}-1$,
such that
\begin{align*}
Q^{(j)}_dT&=
S(T)\\
&+S^{j,d}_{n_{j,d}-1}(T)h^{j,d}_{n_{j,d}-1}(T_{c_j})\\
&\cdots\\
&+S^{j,d}_{1}(T)h^{j,d}_{1}(T_{c_j})\\
&+S^{j,d}_{0}(T).
\end{align*}
Here $S$ represents the extension of the map
$A\longmapsto S^{j,d}_{n_{j,d}}(A)h^{j,d}_{n_{j,d}}(T_{c_j})$
from Lemma \ref{polynomial decomposition}
to the algebra $Q^{(j)}_d
\mathcal{T}_\lambda^{\theta}$.
\end{cor}
\begin{proof}
By the previous lemma, the maps $S^{j,d}_{\ell}$ are continuous and
so we obtain the formula from the lemma by approximating $T$ with elements of
the dense subalgebra $\widetilde{\mathcal{T}}_\lambda^{\theta}$.

By Lemma \ref{polynomial decomposition}, we clearly have
$S(T)\in\operatorname{clos}(\widetilde{\mathcal{T}}_\lambda^{\theta}\cap \operatorname{Rad}(\mathcal{T}_\lambda^{\theta}))$.
Thus we only need to show that
$S^{j,d}_\ell(T)\in\operatorname{Rad}(\mathcal{T}_\lambda^{\theta-\epsilon_j})$,
$\ell=1,\ldots,n_{j,d}-1$.
Note that for such $\ell$ we have $Q^{(j)}_dS^{j,d}_{\ell}(T)=S^{j,d}_{\ell}(T)$.
In particular, $\psi(S^{j,d}_{\ell}(T))=0$ for every functional $\psi=(\mu,\zeta)\in M(\cT_\lambda(\mathbf{d}))$ such that $\mu$ is reached by a net $(\ka^\alpha)$
such that $\ka^\alpha_j$ does not converge to $d$.
Hence it suffices to prove
that $\psi(S^{j,d}_{\ell}(T))=0$
where $\psi=(\mu,\zeta)\in M(\cT_\lambda(\mathbf{d}))$ and $\mu$
is reached by a net $(\ka^\alpha)$ with $\ka^\alpha=d$.
Denote by $M_{j,d}$ the set of all such multiplicative functionals $\psi$.
By Theorem \ref{Maximal ideal space Tdx}, if $(\mu,\zeta)\in M_{j,d}$,
then $\zeta_j\in\spc(\mathbf{T_j}|_{\mathbf{H}_{d}^{(j)}})$.

We proceed by induction.
Consider first the case $\ell=0$.
Let $\psi=(\mu,\zeta)\in M_{j,d}$
and let $\zeta'$ be equal $\zeta$ with the possible exception that
$\zeta'_j=\zeta^{j,d}_1$. Then $h^{j,d}_{\ell}(\zeta'_j)=0$,
for any $\ell\in\{1,\ldots,n_{j,d}\}$, and thus for the functional $\psi'=(\mu,\zeta')$ we have
$\psi'(S^{j,d}_0(T))=\psi'(Q^{(j)}_dT)=0$.

Since $S^{j,d}_0(T)\in\mathcal{T}_\lambda^{\theta-\epsilon_j}$,
the operator $S^{j,d}_0(T)$
does not depend on the operator $T_{c_j}$.
Then $\psi(S^{j,d}_0(T))$ does not depend on $\zeta_j$ and hence $\psi(S^{j,d}_0(T))=\psi'(S^{j,d}_0(T))=0$.

Suppose we have proved that
$\psi(S^{j,d}_{\ell}(T))=0$ for every $\psi\in M_{j,d}$ and
$0\leq\ell< L<n_{j,d}-1$.

Then, by subtracting the first $L-1$ summands in the expression for $Q^{(j)}_dT$,
we see that the operator
\[
T_L=S(T)+S^{j,d}_{n_{j,d}-1}(T)h^{j,d}_{n_{j,d}-1}(T_{c_j})
+\cdots+
S^{j,d}_{L}(T)h^{j,d}_{L}(T_{c_j})
\]
belongs to $\operatorname{Rad}(\mathcal{T}_\lambda^{\theta})$.

Let $\psi=(\mu,\zeta)\in M_{j,d}$ and define $\psi'=(\mu,\zeta')$
where $\zeta'$ equals $\zeta$
with the possible difference that $\zeta'_j=\zeta^{j,d}_{L+1}$.
Therefore we have
\[
(\zeta^{j,d}_{L+1}-\zeta^{j,d}_L)\cdots(\zeta^{j,d}_{L+1}-\zeta^{j,d}_1)\psi'(S^{j,d}_L(T))=\psi'(T_L)=0.
\]

Since, once again, $\psi(S^{j,d}_L(T))$ does not depend on $\zeta_j$,
we conclude that $\psi(S^{j,d}_L(T))=\psi'(S^{j,d}_L(T))=0$.
\end{proof}

Now we are able to prove the main result of this section.
\begin{thm}\label{radical density}
Let $\theta\in\Theta$.
Then $\widetilde{\mathcal{T}}_\lambda^{\theta}\cap\operatorname{Rad}(\mathcal{T}_\lambda^{\theta})$ is dense in
$\operatorname{Rad}(\mathcal{T}_\lambda^{\theta})$.
\end{thm}
\begin{proof}
We proceed by induction on $|\theta|$.

Suppose $|\theta|=0$.
In this case $\mathcal{T}_\lambda^{\theta}=\mathcal{T}_\lambda(L^\infty_{k-qr})$,
which is a $C^*$-algebra and then $\operatorname{Rad}(\mathcal{T}_\lambda^{\theta})=\{0\}$,
from which the claim follows.

As induction hypothesis assume that
$\widetilde{\mathcal{T}}_\lambda^{\theta}\cap\operatorname{Rad}(\mathcal{T}_\lambda^{\theta})$
is dense in
$\operatorname{Rad}(\mathcal{T}_\lambda^{\theta})$ for every $\theta\in\Theta$ such that $|\theta|\leq m'<m$.

Let
$\theta\in\Theta$ with $|\theta|=m'+1$ and let $T\in\operatorname{Rad}(\mathcal{T}_\lambda^{\theta})$.
By Lemma \ref{Rad subset J},
we have $T\in\mathcal{J^\theta}$ and therefore there
is some operator $T'\in\mathcal{J^\theta}$ with
$\|T-T'\|<\varepsilon$
and
\[
T'=
\sum_{j\in J^\theta}
\widetilde{Q}^{(j)}_{\ka_j}S_j,
\]
for some operators $S_j\in\mathcal{L}(\mathcal{A}_\lambda^2(\mathbb{B}^{n}))$
and $\ka_j\in\bZ_+$.

Let $P_j$, $j=1,\ldots,m$, be the projections
defined by \eqref{projections QS def} and \eqref{projections QS def2}. By Lemma \ref{projections QS} we have $T'=\sum_{j\in J^\theta}P_j T'$.
Define now the projection $P_0=I-\sum_{j\in J^\theta}P_j$.
Note that this is the orthogonal projection from $\mathcal{A}_\lambda^2(\mathbb{B}^{n})$
onto the orthogonal complement of the space generated by the images of
the operators $\widetilde{Q}^{(j)}_{\ka_j}$, $j\in J^\theta$.

Then we have
\[
T'=(I-P_0)T'\quad\text{and}\quad P_0T'=0,
\]
and thus
\[
\|T-(I-P_0)T\|=\|P_0T\|=\|P_0(T-T')\|\leq\|T-T'\|<\varepsilon.
\]

Therefore, it suffices to show that
\[
(I-P_0)T=\sum_{j\in J^\theta}P_j
T\in\operatorname{clos}\left(\widetilde{\mathcal{T}}_\lambda^{\theta}\cap\operatorname{Rad}(\mathcal{T}_\lambda^{\theta})\right).
\]

Since the projections $P_j$ are mutually orthogonal,
each operator $P_j T$ is in the radical and, by the definition of $P_j$,
can be written in the form $A\widetilde{Q}^{(j)}_{\ka_j}T$ for some
$A\in \widetilde{\mathcal{T}}_\lambda^{\theta}$.
Thus we are done if we show that
\[
Q^{(j)}_{d}T\in\operatorname{clos}\left(\widetilde{\mathcal{T}}_\lambda^{\theta}\cap\operatorname{Rad}(\mathcal{T}_\lambda^{\theta})\right),
\]
for every $0\leq d \leq \ka_j$ and $j\in J^\theta$.

Fix some $j$ and $d$ as above.
By Corollary \ref{polynomial decomposition 2}, we have
\[
Q^{(j)}_dT=S(T)+S^{j,d}_{n_{j,d}-1}(T)h^{j,d}_{n_{j,d}-1}(T_{c_j})
+\cdots+
S^{j,d}_{1}(T)h^{j,d}_{1}(T_{c_j})
+S^{j,d}_{0}(T),
\]
where $S(T)\in\operatorname{clos}(\widetilde{\mathcal{T}}_\lambda^{\theta}\cap\operatorname{Rad}(\mathcal{T}_\lambda^{\theta}))$ and
$S^{j,d}_{\ell}(T)\in\operatorname{Rad}(\mathcal{T}_\lambda^{\theta-\epsilon_j})$
for every $\ell=0,\ldots,n_{j,d}-1$.
Since $|\theta-\epsilon_j|=m'$, by induction hypothesis
the operators $S^{j,d}_{\ell}(T)$ lie in the algebra
\[
\operatorname{Rad}(\mathcal{T}_\lambda^{\theta-\epsilon_j})
=
\operatorname{clos}\left(\widetilde{\mathcal{T}}_\lambda^{\theta-\epsilon_j}\cap\operatorname{Rad}(\mathcal{T}_\lambda^{\theta-\epsilon_j})\right)
\subset
\operatorname{clos}\left(\widetilde{\mathcal{T}}_\lambda^{\theta}\cap\operatorname{Rad}(\mathcal{T}_\lambda^{\theta})\right).\qedhere
\]
\end{proof}

Taking $\theta=\mathbf{1}$, we get
\begin{cor}
The algebra $\widetilde{\mathcal{T}(\mathbf{d})}\cap\operatorname{Rad}(\mathcal{T}(\mathbf{d}))$
is dense in $\operatorname{Rad}(\mathcal{T}(\mathbf{d}))$.
\end{cor}

We recover as particular cases the characterizations of the radical for
the Banach algebras that were studied in
\cite{Bauer2012}, \cite{Bauer2015} (for the case $m=1$) and \cite{Rodriguez2021}, and
we obtain the corresponding missing result in
\cite{Garcia2015}.
Furthermore,
changing some notation and details, our method
can be adapted without much effort
to obtain an analogous characterization for the Banach algebras
generated by Toeplitz operators with quasi-radial quasi-homogeneous symbols
(for a general $m\geq1$),
giving thus a solution to the open problem mentioned at the end of Section 6 of
\cite{Bauer2015}.
In all these cases, the radical equals the closure of its intersection with
the non-closed dense subalgebra of finite sums of finite products of the generators.